\documentclass[hidelinks,onefignum,onetabnum]{siamart251216}

\usepackage{lipsum}
\usepackage{amsfonts}
\usepackage{graphicx}
\usepackage{epstopdf}
\usepackage{algorithmic}
\ifpdf
  \DeclareGraphicsExtensions{.eps,.pdf,.png,.jpg}
\else
  \DeclareGraphicsExtensions{.eps}
\fi

\newsiamremark{remark}{Remark}
\newsiamremark{hypothesis}{Hypothesis}
\crefname{hypothesis}{Hypothesis}{Hypotheses}
\newsiamthm{claim}{Claim}
\newsiamremark{fact}{Fact}
\crefname{fact}{Fact}{Facts}
\newsiamremark{assumption}{Assumption}
\crefname{assumption}{Assumption}{Assumptions}

\headers{Characterizing Optimality Robustness}{Y. Lin, Z. Duan, T. Li, N. Bai, and Z. Sun}

\title{Characterizing Robustness in Nonlinear Optimal Control: From Stability to Optimality
\thanks{Submitted to the editors August 5, 2026. A related preprint by some of the authors is available at arXiv:2604.05633. It develops an analysis-to-design framework for Koopman-based data-driven control, with the main focus on robustness-aware controller design and its computational implementation, whereas the present manuscript studies the characterization of optimality robustness itself from a broader nonlinear optimal control perspective.
\funding{This work was funded by the National Natural Science Foundation of China (NSFC) under grants T2121002, U24A20266, 62173006 and 62533017.}}}

\author{Yicheng Lin\thanks{School of Advanced Manufacturing and Robotics, Peking University, Beijing, 100871 China 
  (\email{linyc020709@stu.pku.edu.cn, duanzs@pku.edu.cn, tlee@stu.pku.edu.cn, zhiyong.sun@pku.edu.cn}).}
\and Zhisheng Duan\footnotemark[2] \and Tianzhi Li\footnotemark[2]
\and Nan Bai\thanks{Department of Electronic and Computer Engineering, The Hong Kong University of Science and Technology, Hong Kong SAR, 999077 China (\email{eenanbai@ust.hk})}\and Zhiyong Sun\footnotemark[2]}

\usepackage{amsopn}

\ifpdf
\hypersetup{
  pdftitle={Characterizing Optimality Robustness},
  pdfauthor={Y. Lin, Z. Duan, T. Li, N. Bai, and Z. Sun}
}
\fi

\begin{document}

\maketitle

\begin{abstract}
In nonlinear optimal control, uncertainties in system dynamics may affect not only closed-loop stability but also the achieved optimality properties of the resulting solutions. This paper develops a systematic robustness analysis for nonlinear optimal control beyond the conventional focus on stability in robust control theory. 
First, we demonstrate that the optimal value function retains its Lyapunov property under a quantifiable criterion, thereby guaranteeing the preservation of closed-loop stability. Building upon this foundation, we establish explicit characterizations for optimality deviations induced by model mismatch in both closed-loop performance and optimal controllers, and further reveal their consistency with classical linear-quadratic regulator (LQR) results. In addition, the robustness analysis admits a unified computational formulation that gives rise to an iterative scheme with guaranteed convergence, enabling quantitative assessment of optimality robustness in nonlinear control systems. Numerical examples validate the theoretical analysis.
\end{abstract}

\begin{keywords}
Robustness analysis, nonlinear optimal control, model uncertainty, nonlinear systems, optimality deviation.
\end{keywords}

\begin{MSCcodes}
49L20, 93C10, 49N10, 93D09
\end{MSCcodes}

\section{Introduction}
Optimal control aims at designing controllers that optimize desired performance objectives, which has long been a fundamental topic in control theory. The resulting optimal solutions are intrinsically determined by system dynamics through the optimality conditions \cite{liberzon2012calculus}. However, such dynamics are often subject to heterogeneous uncertainties including imperfect knowledge, modeling inaccuracies, and external disturbances \cite{hewing2020learning, deisenroth2015gp, lin2025integratinguncertainties}, which may affect the optimality properties of the resulting optimal solutions. This naturally motivates a critical question, namely, \textbf{how robustness should be understood in nonlinear optimal control beyond the conventional focus on stability in robust control theory}.

Robustness has long been recognized as a central concern in control theory. Classical robust control theories have primarily focused on preserving closed-loop stability and guaranteeing prescribed performance requirements in the presence of various uncertainties \cite{zhou1996robust}. These developments have established systematic analysis and synthesis approaches for robust stabilization and robust performance optimization \cite{khalil1996nonlinear, doyle1989state}. In the context of optimal control, existing studies have further incorporated uncertainty into the controller design process through various robust formulations, including min-max control, $H_\infty$ optimization, risk-sensitive control, and guaranteed-cost control \cite{aliyu2011nonlinear, ugrinovskii1999absolute, fleming1995risk}. Nevertheless, these approaches mainly address how to design controllers that achieve certain performance indices under uncertainty, while the characterization of how uncertainty affects the optimality property of a controller derived from an imperfect system description remains insufficiently understood.

In the linear setting, related questions admit clean answers in classical robust control theory \cite{zhou1996robust}, most notably through perturbation analysis of the algebraic Riccati equation (ARE) \cite{kenney1988sensitivity, konstantinov1993perturbation}, which yields explicit first-order sensitivity of the optimal performance to the modeling uncertainty. By contrast, in the nonlinear setting, one of the central challenges for such characterizations stems from the implicit nonlinear coupling between the optimal performance and the resulting controller through the optimality conditions. For instance, the associated Hamilton-Jacobi-Bellman (HJB) equation for nonlinear systems generally does not admit an explicit perturbation structure available in ARE. As a consequence, the impact of model mismatch on optimality becomes highly nontrivial to characterize.

This paper addresses the above challenges by establishing a \textit{systematic} characterization of robustness in nonlinear optimal control. Building upon the robust stability guarantees, we explicitly characterize the optimality deviation induced by model mismatch, which ultimately enables a \textit{quantitative} evaluation of optimality robustness with explicit connections to classical linear optimal control results. The main contributions are summarized as follows.
\begin{itemize}
	\item \textbf{Robust stability foundation:} We demonstrate that the nominal optimal value function remains valid as a Lyapunov candidate under the robust stability criterion, thereby preserving closed-loop stability. The resulting criterion further provides a quantitative guideline for the surrogate modeling accuracy to satisfy the fundamental requirement of robust stability.
	\item \textbf{Optimality deviation quantification:} We establish explicit characterizations of optimality deviations in both closed-loop performance and the resulting controller, revealing how model mismatch propagates through the optimal control structure. The analysis is developed in a hierarchical manner, where the performance deviation is evaluated for general nonlinear systems, while the controller deviation is further characterized explicitly for the control-affine dynamics. Moreover, the proposed framework consistently recovers classical sensitivity results in the linear quadratic regulator (LQR) setting, thereby establishing a theoretical bridge between nonlinear optimality robustness and ARE perturbation analysis.
	\item \textbf{Computational robustness evaluation:} We show that the optimality deviation analysis admits a unified computational formulation, where seemingly different deviation quantities can be reduced to a class of PDE-solving problems. Based on this formulation, an iterative algorithm with guaranteed convergence is developed for quantitative evaluation of optimality robustness.
\end{itemize}

\textit{Outline.} This paper is organized as follows. \Cref{2.preliminaries} introduces the problem formulation. Sections~\ref{3.stability} and \ref{4.optimality} investigate the issues of stability preservation and optimality deviation, respectively. \Cref{5.linear} exemplifies the proposed robustness analysis through the classical LQR framework, while \Cref{6.computation} develops the iterative algorithm for unified computation of optimality deviation together with numerical examples. \Cref{7.conclusion} concludes this paper.

\textit{Notations.} Throughout this paper, we denote by $\mathbb{R}^n$ the $n$-dimensional Euclidean space. Unless otherwise noted, the norm for a real vector $v\in\mathbb{R}^{n_v}$ is the Euclidean norm $\|v\|=\sqrt{\sum_{i=1}^{n_v}v_i^2} =\sqrt{v^\top v}$, meanwhile the norm for a real matrix $M=(m_{ij})\in\mathbb{R}^{n_r\times n_c}$ is the induced 2-norm which equals to the maximum singular value. The null and column spaces of $M$ are respectively denoted by $\mathrm{Nul}(M)$ and $\mathrm{Col}(M)$, while the direct sum of two subspaces is denoted by $\oplus$. For a symmetric matrix $M$, we denote by $\lambda_{\min}(M)$ and $\lambda_{\max}(M)$ its minimum and maximum eigenvalue, respectively. For two symmetric matrices $M_1,M_2$, $M_1 \succeq(\preceq) M_2$ means that matrix $M_1-M_2$ is positive (negative) semidefinite, i.e., $M_1-M_2\succeq(\preceq)0$. $\mathcal{C}^0(\cdot)$ denotes the space of continuous functions defined on the corresponding domains, and $\mathcal{C}^k(\cdot)$ the space of functions with continuous derivatives up to order $k$.

\section{Problem formulation}\label{2.preliminaries}
Consider a continuous-time nonlinear system represented by
\begin{equation}\label{original system}
	\dot{x}(t)=f\left(x(t),u(t)\right), \quad x(0)=x,
\end{equation}
where $x(t)\in\mathbb{X}\subseteq\mathbb{R}^n$ is the state, $u(t)\in\mathbb{U}\subseteq \mathbb{R}^m$ is the control input. We focus our analysis on a compact state space $\mathbb{X}$ containing the origin, which is assumed to be a forward-invariant domain under the optimal controllers to be discussed. The system dynamics $f:\mathbb{X}\times\mathbb{U}\rightarrow \mathbb{X}$ is $\mathcal{C}^2$, and the origin is an equilibrium of the unactuated system, i.e., $f(0,0)=0$. The objective of the optimal control problem is to minimize an infinite-horizon cost functional
\begin{equation}\label{optimal control problem}
	J(x_0,u(\cdot))=\int_{0}^{\infty}l(x(t),u(t)) \mathrm{d}t.
\end{equation}
The actual optimal value function is defined as
\begin{equation}
	V^*(x)=\inf_{u(\cdot)}J(x,u(\cdot)),
\end{equation}
and the corresponding actual optimal controller is denoted by $u^*(x)$. According to the classical optimal control theory~\cite{1995Linear}, $V^*(x)$ should satisfy the HJB equation
\begin{equation}\label{actual HJB}
	\inf_{u(\cdot)}\left\{\nabla V^*(x)f(x,u)+l(x,u)\right\}=0,
\end{equation}
where $\nabla V^*=\frac{\partial}{\partial x}V^*$ denotes the gradient of $V^*$.

In practice, optimal controllers are often synthesized based on an approximate surrogate model of the system dynamics, which may not exactly represent the actual dynamics due to modeling errors, uncertain parameters, or incomplete knowledge. Such model mismatch naturally arises in various scenarios, including first-principles modeling, data-driven identification methods, and learning-based approaches \cite{PNAS2016SINDy, lewis1999nncontrol, mauroy2020koopman}, and may affect the optimal solutions. Given a surrogate model, the robustness analysis developed in this paper is independent of the specific modeling procedure.

Accordingly, we consider a nominal surrogate model represented by
\begin{equation}\label{surrogate model}
	\dot{x}(t)=\hat{f}(x(t),u(t)).
\end{equation}
Since the surrogate model is only an approximation of the true dynamics, the actual system can be equivalently written as
\begin{equation}\label{actual model}
	\dot{x}(t)=\hat{f}(x(t),u(t))+r(x(t),u(t)),
\end{equation}
where the approximation error $r(x,u)$ represents the inevitable model mismatch induced by the surrogate modeling process. Based on the surrogate model \eqref{surrogate model}, the nominal optimal controller $u_0^*$ can be synthesized using various existing optimal control methods. The corresponding nominal value function $V_0^*$ is governed by
\begin{equation}\label{nominal HJB}
	\inf_{u(\cdot)}\left\{\nabla V_0^*(x)\hat{f}(x,u)+l(x,u)\right\}=0.
\end{equation}

The mismatch between nominal and actual systems leads to deviations between the corresponding optimal solutions, including the value functions $V^*$ and $V_0^*$ as well as the controllers $u^*$ and $u_0^*$. The subsequent robustness analysis aims to explicitly characterize such deviations. To this end, the following assumptions are imposed.
\begin{assumption}\label{Assum: running cost}
	The nominal dynamics $\hat{f}:\mathbb{X}\times\mathbb{U}\rightarrow \mathbb{X}$ in the surrogate model \eqref{surrogate model} and the running cost $l:\mathbb{X}\times\mathbb{U} \rightarrow\mathbb{R}_{\geq 0}$ in the functional~\eqref{optimal control problem} are $\mathcal{C}^2$, $\hat{f}(0,0)=0,\ l(0,0)=0$, and there exist constants $\alpha_1>0,\alpha_2>0$ such that
	\begin{equation}\label{running cost lower bound}
		l(x,u)\geq \alpha_1\|x\|^2+\alpha_2\|u\|^2,\quad\forall(x,u)\in\mathbb{X}\times\mathbb{U}.
	\end{equation}
\end{assumption}
\begin{remark}[Generality of cost]
	Assumption~\ref{Assum: running cost} preserves the positive-definite and coercive structure of the running cost $l$. Particularly, the standard form $l(x,u)=x^\top Qx+u^\top Ru$ in the LQR setting naturally satisfies \eqref{running cost lower bound} for $Q,R\succ0$, meanwhile more general nonlinear running costs with higher-order growth are also included. As a result, Assumption~\ref{Assum: running cost} serves as a natural extension of classical linear optimal control theory to nonlinear settings. Such a condition ensures sufficient penalization of large states and control inputs.
\end{remark}
\begin{assumption}\label{Assum: value functions}
	The value functions $V^*,V_0^*:\mathbb{X}\rightarrow\mathbb{R}_{\geq 0}$ given by \eqref{actual HJB} and \eqref{nominal HJB} respectively are classical solutions such that $V^*,V_0^*\in\mathcal{C}^2(\mathbb{X})$. Further, it is assumed that $u^*$ is the unique minimizer for \eqref{actual HJB}, and analogously for $u_0^*$ in \eqref{nominal HJB}.
\end{assumption}
\begin{remark}[Regularity of value function]
	Under this standing assumption, our focus is to characterize the impacts of uncertainties on stability and optimality, rather than on analyzing the existence and uniqueness of solutions. In addition, for the class of nonlinear optimal control problems considered here, the existence of a unique $\mathcal{C}^2$ solution to the HJB equation \eqref{actual HJB} or \eqref{nominal HJB} in a neighborhood of the origin is ensured by~\cite{lukes1969optimal}, provided that the system is locally stabilizable and the running cost satisfies \eqref{running cost lower bound}. In more general settings, the value functions can be interpreted in the viscosity sense \cite{bardi1997optimal}, and the results in this paper can be extended to weaker notions of solutions with additional technical efforts in the future.
\end{remark}
\begin{assumption}\label{Assum: error bound}
	The approximation error term in \eqref{actual model} is bounded by
	\begin{equation}\label{error bound}
		\|r(x,u)\|\leq r_1(x,u)\|x\|+r_2(x,u)\|u\|,
	\end{equation}
	where $r_1,r_2$ are continuous and uniformly bounded in $\mathbb{X}\times\mathbb{U}$, i.e., $\|r_1(x,u)\|\leq c_1, \|r_2(x,u)\|\leq c_2$.
\end{assumption}
\begin{remark}[On the uncertainty bound]
	Assumption~\ref{Assum: error bound} imposes a (semi-)linear growth bound on the approximation error with respect to the state and control input. From a theoretical viewpoint, it naturally arises from the local smoothness of both the actual and nominal dynamics given Assumption~\ref{Assum: running cost}. In particular, since $r(0,0)=f(0,0) - \hat{f}(0,0)$, continuous and bounded Jacobians $\frac{\partial r}{\partial x},\frac{\partial r}{\partial u}$ over a compact operating region $\mathbb{X}\times\mathbb{U}$ imply the state- and input-dependent uncertainty bound in \eqref{error bound}. Such a characterization is consistent with the vanishing perturbation condition in nonlinear robust control \cite{khalil1996nonlinear}.
	
	From a practical perspective, the explicit characterization of model mismatch has received increasing attention in modern modeling and control, enabling rigorous robustness analysis beyond empirical model fitting. For instance, proportional state- and input-dependent error bounds have been established for Koopman-based surrogate models~\cite[Proposition 5]{Strasser2024Koopman}, while deterministic error bounds have been derived for kernel-based learning models as uncertainty certificates~\cite[Theorem 1]{maddalena2021deterministic}. Although the specific forms of such bounds vary with different modeling approaches, they provide a common basis for quantifying approximation errors in surrogate models. Motivated by these developments, Assumption~\ref{Assum: error bound} adopts a growth-bound description that is suitable for robustness analysis in nonlinear optimal control problem.
\end{remark}

\section{Robustness of stability}\label{3.stability}
Under the nominal settings where $r(x,u)=0$, the nominal optimal value function $V_0^*$ serves as a Lyapunov function, since the HJB optimality condition together with Assumption~\ref{Assum: running cost} yields $\dot{V}_0(x)=-l(x,u_0^*)<0$. However, when the nominal optimal controller $u_0^*$ is applied to the actual system~\eqref{actual model}, the additional approximation error term may destroy this stabilizing property. Therefore, before investigating the robustness of optimality, it is necessary to characterize whether the closed-loop stability can be preserved under model mismatch.

The norm-bounded characterization \eqref{error bound} enables a tractable Lyapunov-based analysis of robust stability, while preserving a general applicability across different techniques. Based on Assumption~\ref{Assum: error bound}, we define the set $\mathcal{R}$ of admissible approximation error satisfying \eqref{error bound}, i.e., 
\begin{equation}
	\mathcal{R}=\{r(x,u)|\|r(x,u)\|\leq c_1\|x\|+c_2\|u\|\}.
\end{equation}

Before investigating the robustness of stability, we introduce the following lemma to characterize a key property of the value function.
\begin{lemma}\label{Lem: gradient of value function}
	Under Assumptions~\ref{Assum: running cost}, \ref{Assum: value functions} and \ref{Assum: error bound}, the gradient of nominal value function $\nabla V_0^*(x)$ is Lipshitz continuous in $\mathbb{X}$. Hence, there exists $L>0$ such that $\| \nabla V_0^*(x)\| \leq L\|x\|,\forall x\in\mathbb{X}$.
\end{lemma}
\begin{proof}
	Since $V_0^*(x)\in\mathcal{C}^2(\mathbb{X})$ by Assumption~\ref{Assum: value functions}, its Hessian matrix $\nabla^2 V_0^*(x)$ is continuous and uniformly bounded on the compact $\mathbb{X}$. Let $L=\text{ess}\sup_{x\in\mathbb{X}}\|\nabla^2 V_0^*(x)\|$. Recalling that with Assumption~\ref{Assum: running cost} for the running cost $l$, $V_0^* (x)\in\mathcal{C}^2(\mathbb{X})$ is a positive definite value function that achieves its minimum at the origin, $\nabla V_0^*(0)=0$ holds.
	
	By the mean value theorem in integral form, $\nabla V_0^*(x)=\int_{0}^{1} \nabla^2V_0^*(tx)\mathrm{d}t\cdot x$ holds. Taking the norm on both sides yields
	\begin{equation*}
		\|\nabla V_0^*(x)\|\leq \int_{0}^{1}\left\|\nabla^2V_0^*(tx)\right\|\mathrm{d}t\cdot\|x\| \leq L\|x\|.
	\end{equation*} The proof is completed.
\end{proof}

With the aid of Lemma~\ref{Lem: gradient of value function}, robustness of stability is established by the following theorem.
\begin{theorem}\label{Thm: Robust Stability}
	Let Assumptions~\ref{Assum: running cost}, \ref{Assum: value functions} and \ref{Assum: error bound} hold. If the error bound coefficients $c_1>0,c_2>0$ satisfy the \textit{robust stability criterion}
	\begin{equation}\label{robust stability criterion}
		c_1L+\frac{c_2^2}{4\alpha_2}L^2<\alpha_1,
	\end{equation}
	then the nominal optimal controller $u_0^*$ given by \eqref{nominal HJB} asymptotically stabilizes the nonlinear system \eqref{actual model} for any $r(x,u)\in\mathcal{R}$. Hence, the nominal optimal controller $u_0^*$ stabilizes the original nonlinear system \eqref{original system}.
\end{theorem}
\begin{proof}
	First, we calculate the time derivative of nominal value function $V_0^*$ along the system \eqref{actual model} controlled by $u_0^*$, i.e.,
	\begin{equation}
		\begin{aligned}
			\frac{\mathrm{d}V_0^*}{\mathrm{d}t}&=\left(\nabla V_0^*\right)^\top\left[\hat{f}(x(t), u_0^*(t))+r(x(t),u_0^*(t))\right] \\&=-l(x(t),u_0^*(t))+\left(\nabla V_0^*\right)^\top r(x(t),u_0^*(t)).
		\end{aligned}
	\end{equation}
	With Assumption~\ref{Assum: error bound} and Lemma~\ref{Lem: gradient of value function}, we obtain
	\begin{equation}
		\begin{aligned}
			\frac{\mathrm{d}V_0^*}{\mathrm{d}t}&\leq-l(x(t),u_0^*(t))+\|\nabla V_0^*\|\left( c_1\|x(t)\|+c_2\|u_0^*(t)\|\right) \\ &\leq -l(x(t),u_0^*(t))+c_1L\|x(t)\|^2+c_2L\|x(t)\|\|u_0^*(t)\|.
		\end{aligned}
	\end{equation}
	Since the inequality
	\begin{equation}
		\|x(t)\|\|u_0^*(t)\|\leq\frac{1}{2}\left(\beta\|x(t)\|^2+\frac{1}{\beta}\|u_0^*(t)\|^2 \right)
	\end{equation}
	holds for any $\beta>0$, the time derivative of $V_0^*$ satisfies
	\begin{equation}\label{time derivative of V0}
		\begin{aligned}
			\frac{\mathrm{d}V_0^*}{\mathrm{d}t}\leq -\alpha_1\|x(t)\|^2-\alpha_2\|u_0^*(t)\|^2 +\left(c_1+\frac{1}{2}\beta c_2\right)L\|x(t)\|^2+\frac{1}{2\beta} c_2L\|u_0^*(t)\|^2.
		\end{aligned}
	\end{equation}
	If there exists a constant $\beta>0$ such that $c_1+\frac{1}{2}\beta c_2<\frac{\alpha_1}{L}$ and $c_2<\frac{2\beta\alpha_2}{L}$, the time derivative of $V_0^*(x)$ is negative definite. With the definition of the value function and Assumption~\ref{Assum: running cost}, $V_0^*(0)=0$ and $V_0^*(x)>0,\forall x\neq 0$ naturally hold. Therefore, $V_0^*$ is a Lyapunov function for any possible nonlinear system \eqref{actual model} satisfying $r(x,u)\in\mathcal{R}$, which indicates that $u_0^*$ stabilizes the original system \eqref{original system}.
	
	Note that
	\begin{subequations}\label{beta condition}
		\begin{align}
			c_1+\frac{1}{2}\beta c_2<\frac{\alpha_1}{L} &\Leftrightarrow \beta<\frac{2\alpha_1-2c_1L}{c_2L}, \\
			c_2<\frac{2\beta\alpha_2}{L} &\Leftrightarrow \beta>\frac{c_2L}{2\alpha_2}.
		\end{align}
	\end{subequations}
	Therefore, there exists $\beta$ satisfying \eqref{beta condition} if and only if \begin{equation}
		\frac{c_2L}{2\alpha_2}< \frac{2\alpha_1-2c_1L}{c_2L},
	\end{equation}which is equivalent to the robust stability criterion \eqref{robust stability criterion}. The proof is completed.
\end{proof}

Theorem~\ref{Thm: Robust Stability} reveals an intrinsic synergy between optimality and stabilization in the presence of model mismatch. It highlights that in the considered infinite-horizon optimal control problem, the pursuit of optimality does not necessarily compromise closed-loop stability under model mismatch. Instead, stability can still be retained under appropriate conditions. As long as the data-driven surrogate model maintains a sufficient degree of accuracy, the nominal optimal controller is capable of preserving closed-loop stability.

From a complementary viewpoint, the robust stability criterion \eqref{robust stability criterion} provides an uncertainty budget for the basic requirement of stability in nonlinear optimal control. It establishes a rigorous threshold for the approximation error bound $(c_1,c_2)$, beyond which the stabilizing property of the nominal optimal controller $u_0^*$ may no longer be guaranteed. In practical control applications, obtaining a sufficiently accurate surrogate model can be challenging due to modeling complexity and inevitable approximation error. As a consequence, Theorem~\ref{Thm: Robust Stability} provides a practical guideline for surrogate modeling, indicating the required accuracy to guarantee closed-loop stability.

Following the same line of analysis, the asymptotic stability of the actual optimal controller can be established as a direct consequence.

\begin{corollary}\label{Coro: V^*}
	Let Assumptions~\ref{Assum: running cost}-\ref{Assum: value functions} hold. Then, the actual optimal controller $u^*$ obtained from \eqref{actual HJB} renders the system~\eqref{original system} asymptotically stable.
\end{corollary}
\begin{proof}
	The proof of Corollary~\ref{Coro: V^*} follows similar thoughts with Theorem~\ref{Thm: Robust Stability}. The time derivative of $V^*$ given by \eqref{actual HJB} along the system \eqref{original system} controlled by $u^*$ satisfies $\dot{V}^*=-l(x(t),u^*(t))$, which holds negative definiteness with Assumption~\ref{Assum: running cost}. Thus, $V^*$ is a Lyapunov function, and the closed-loop system \eqref{original system} controlled by $u^*$ is asymptotically stable.
\end{proof}

\section{Robustness of optimality}\label{4.optimality}
While Section~\ref{3.stability} establishes conditions for preserving closed-loop stability, the achieved performance may deviate from the true optimum for the actual system under model mismatch. Such discrepancy in the achieved optimality is referred to in this paper as \textit{optimality deviation}, a general concept that encompasses both performance loss and controller deviation.

Accordingly, we shift our focus from stability to optimality, quantifying how model mismatch affects the optimal solutions in two complementary steps. Firstly, we analyze the optimality deviation in the value function, which captures performance loss induced by model mismatch. Secondly, we study the optimality deviation in the optimal controller, which reflects how the resulting feedback is perturbed. The two-step analysis is motivated by the intrinsic structure of optimal control, where the optimal controller is implicitly characterized through the value function via the HJB equation. Understanding perturbations in the value function thus provides the foundation for analyzing deviations in the controller.

\subsection{Deviation of performance}\label{4.1 performance deviation}
Since the approximation error $r(x,u)$ perturbs the system behavior, it affects the closed-loop performance \eqref{optimal control problem} under the same controller $u(\cdot)$. To facilitate the subsequent analysis, we use the notation $J(u,x,r)$. Apparently, $V_0^*(x)=J(u_0^*,x,0)$ and $V^*(x)=J(u^*,x,r)$ hold by their definitions. We further define a value function $V(x)=J(u_0^*, x,r)$, which characterize closed-loop performance \eqref{optimal control problem} of the actual nonlinear system \eqref{original system} controlled by the nominal optimal controller~$u_0^*$.

First, we consider the optimality deviation of nominal value function $V_0^*(x)$ with the following theorem. We note that the gradient-energy integral term will be rigorously justified to be finite after the theorem.
\begin{theorem}\label{Thm: performance deviation}
	Due to the model mismatch, applying the nominal optimal controller $u_0^*$ given by \eqref{nominal HJB} to the actual nonlinear system \eqref{original system} leads to an extra cost in \eqref{optimal control problem}. With Assumptions~\ref{Assum: running cost}, \ref{Assum: value functions} and \ref{Assum: error bound}, this extra cost is characterized by
	\begin{equation}\label{optimal performance deviation}
		\begin{aligned}
			\left\|V-V_0^*\right\|\leq \Delta V_{\max}=&\frac{1}{2}C_{12}^2\int_{0}^\infty\|\nabla V_0^*\|^2 \mathrm{d}t \\
			&+\frac{C_{12}}{2}\left(\int_{0}^\infty\|\nabla V_0^*\|^2 \mathrm{d}t\right)^{\frac{1}{2}}\sqrt{C_{12}^2\int_{0}^\infty\|\nabla V_0^*\|^2 \mathrm{d}t+4V_0^*}.
		\end{aligned}
	\end{equation}
	where $C_{12}=\max\left\{\sqrt{\frac{2}{\alpha_1}}c_1, \sqrt{\frac{2}{\alpha_2}}c_2\right\}$, and the integral term of $\|\nabla V_0^*\|^2$ is evaluated along the trajectory controlled by $u_0^*$ under the worst-case error $r_0^*$ given by \eqref{nominal worst-case error}, i.e.,
	\begin{equation}\label{nominal worst trajectory u0}
		\dot{x}(t)=\hat{f}(x(t),u_0^*(t))+r_0^*(x(t),u_0^*(t)),\ x(0)=x.
	\end{equation}
\end{theorem}
\begin{proof}
	In order to bound $V-V_0^*$, we firstly investigate the worst-case approximation error $r_0^*=\arg\max_{r\in\mathcal{R}}J(u_0^*,x,r)$ that aims to maximize the cost functional \eqref{optimal control problem}. The corresponding worst-case value function over all admissible errors is denoted by $V_r^*(x)=\max_{r\in\mathcal{R}}J(u_0^*,x,r)$. Using the Hamilton-Jacobi-Issacs (HJI) equation~\cite{aliyu2011nonlinear}, $r_0^*$ and $V_r^*$ should be solved from
	\begin{equation}\label{nominal worst-case maximum}
		\begin{aligned}
			\max_{r\in\mathcal{R}}\left\{(\nabla V_r^*)^\top\left[\hat{f}(x,u_0^*)+r(x,u_0^*) \right] +l(x,u_0^*)\right\}=0.
		\end{aligned}
	\end{equation}
	Since the inner product $(\nabla V_r^*)^\top r$ is linear in $r$ and $\mathcal{R}$ is norm-bounded under Assumption~\ref{Assum: error bound}, the maximization is achieved when $r\in\mathcal{R}$ is aligned with $\nabla V_r^*$, yielding
	\begin{equation}\label{nominal worst-case error}
		r_0^*(x,u_0^*)=\dfrac{c_1\|x\|+c_2\|u_0^*\|}{\|\nabla V_r^*\|}\nabla V_r^*.
	\end{equation}
	With \eqref{nominal worst-case maximum} and \eqref{nominal worst-case error},  $V_r^*(x)=J(u_0^*,x,r_0^*)$ is the solution of
	\begin{equation}\label{actual HJI}
		\begin{aligned}
			(\nabla V_r^*)^\top\hat{f}(x,u_0^*)+l(x,u_0^*)=-(c_1\|x\|+c_2\|u_0^*\|) \|\nabla V_r^*\|.
		\end{aligned}
	\end{equation}
	Note that for the degenerate case in \eqref{nominal worst-case error} where $\|\nabla V_r^*\|=0$, we obtain $0=l(x,u_0^*)\geq \alpha_1\|x\|^2+\alpha_2\|u_0^*\|^2$ with Assumption~\ref{Assum: running cost}, which only occurs at the origin $(x,u_0^*)=(0,0)$. Thus, we can impose $r_0^*(0,0)=0$ to ensure that $r_0^*$ is continuous and well-defined in $\mathbb{X}\times\mathbb{U}$.
	Consequently, the time derivative of $V_r^*-V_0^*$ along \eqref{nominal worst trajectory u0} is
	\begin{equation*}
		\begin{aligned}
			\frac{\mathrm{d}}{\mathrm{d}t}(V_r^*-V_0^*)=&\nabla(V_r^*)^\top \left[\hat{f}(x(t),u_0^*(t))+r_0^*(x(t),u_0^*(t))\right]\\
			&-\nabla(V_0^*)^\top \hat{f}(x(t),u_0^*(t))-\nabla(V_0^*)^\top r_0^* \\
			=&-\nabla(V_0^*)^\top \dfrac{c_1\|x(t)\|+c_2\|u_0^*(t)\|}{\|\nabla V_r^*\|}\nabla V_r^*,
		\end{aligned}
	\end{equation*}
	which is obtained by combining \eqref{nominal HJB} and \eqref{actual HJI}, then
	\begin{equation*}
		\left\|\frac{\mathrm{d}}{\mathrm{d}t}(V_r^*-V_0^*)\right\|\leq\|\nabla V_0^*\|\left(c_1\|x(t)\| +c_2\|u_0^*(t)\|\right).
	\end{equation*}
	Since the closed-loop stability has been proven in Theorem~\ref{Thm: Robust Stability} and hence $x(\infty)=0$, $V_r^*$ and $V_0^*$ tend to be zero when $t\rightarrow\infty$. Taking the norm from both sides of
	\begin{equation*}
		V_r^*(x)-V_0^*(x)=-\int_{0}^\infty\frac{\mathrm{d}}{\mathrm{d}t}\left[V_r^*(x(t))-V_0^*(x(t))\right]\mathrm{d}t,
	\end{equation*}
	and noting that $V_r^*(x)-V_0^*(x)\geq 0$ holds natively due to the maximizing property of the worst-case approximation error $r_0^*$, we obtain
	\begin{equation}\label{combine 1}
		\begin{aligned}
			V_r^*-V_0^*\leq& \int_{0}^\infty \left\|\frac{\mathrm{d}}{\mathrm{d}t}(V_r^*-V_0^*)\right\|\mathrm{d}t \\
			\leq&\left(\int_{0}^\infty\|\nabla V_0^*\|^2\mathrm{d}t\right)^{ \frac{1}{2}}\left(\int_{0}^\infty(c_1\|x(t)\|+c_2\|u_0^*(t)\|)^2\mathrm{d}t\right)^{\frac{1}{2}}.
		\end{aligned}
	\end{equation}
	The last integral term corresponding to the approximation error can be bounded by
	\begin{equation}\label{combine 2}
		\begin{aligned}
			&\int_{0}^\infty\left(c_1\|x(t)\|+c_2\|u_0^*(t)\|\right)^2\mathrm{d}t
			\leq\int_{0}^\infty2\left(c_1^2\|x(t)\|^2+c_2^2\|u_0^*(t)\|^2\right)\mathrm{d}t\\
			\leq&\int_{0}^\infty\max\left\{\frac{2c_1^2}{\alpha_1},\frac{2c_2^2}{\alpha_2}\right\} \left(\alpha_1\|x(t)\|^2+\alpha_2\|u_0^*(t)\|^2\right) \mathrm{d}t\\
			\leq&\max\left\{\frac{2c_1^2}{\alpha_1},\frac{2c_2^2}{\alpha_2}\right\} \int_{0}^\infty l(x(t),u_0^*(t))\mathrm{d}t=C_{12}^2V_r^*(x).
		\end{aligned}
	\end{equation}
	Combining \eqref{combine 1} and \eqref{combine 2} induces
	\begin{equation}\label{optimal performance deviation version 2}
		\begin{aligned}
			V_r^*-V_0^*\leq C_{12}\left(\int_{0}^\infty\|\nabla V_0^*\|^2 \mathrm{d}t\right)^{\frac{1}{2}}(V_r^*)^\frac{1}{2},
		\end{aligned}
	\end{equation}
	which can be regarded as a quadratic inequality in the nonnegative variable $(V_r^*)^{\frac{1}{2}}$. Solving the corresponding quadratic equation and retaining its nonnegative root yields
	\begin{equation}\label{combine 3}
		\begin{aligned}
			(V_r^*)^\frac{1}{2}\leq \frac{1}{2}C_{12}\left(\int_{0}^\infty\|\nabla V_0^*\|^2 \mathrm{d}t\right)^{\frac{1}{2}} +\frac{1}{2}\sqrt{C_{12}^2\int_{0}^\infty\|\nabla V_0^*\|^2 \mathrm{d}t+4V_0^*}.
		\end{aligned}
	\end{equation}
	Since $r_0^*$ denotes the worst-case approximation error that tries to maximize the value function, $\|V-V_0^*\|\leq V_r^*-V_0^*$ holds, and the upper bound \eqref{optimal performance deviation} can be verified by calculating the square of both sides of \eqref{combine 3}. By symmetry, when the approximation error acts in the opposite direction, the resulting value function may also be lower than $V_0^*$. Under this circumstance, $r_0^{**}=\arg\min_{r\in\mathcal{R}} J(u_0^*,x,r)$ should be solved from the HJB equation
	\begin{equation*}
		\min_{r\in\mathcal{R}}\left\{(\nabla V_r^{**})^\top\left[\hat{f}(x,u_0^*)+r(x,u_0^*) \right]+l(x,u_0^*)\right\}=0.
	\end{equation*}
	The minimization of $(\nabla V_r^{**})^\top r$ linearly dependent on $r$ is also achieved when $r\in\mathcal{R}$ is aligned with $\nabla V_r^*$, yielding
	\begin{equation*}
		r_0^{**}=-\dfrac{c_1\|x\|+c_2\|u_0^*\|}{\|\nabla V_r^{**}\|}\nabla V_r^{**}.
	\end{equation*}
	The time derivative can be obtained
	\begin{equation*}
		\frac{\mathrm{d}}{\mathrm{d}t}(V_r^{**}-V_0^*)=\nabla(V_0^*)^\top \dfrac{c_1\|x\|+c_2\|u_0^*\|}{\|\nabla V_r ^{**}\|}\nabla V_r^{**},
	\end{equation*}
	and the same upper bound of $\|V_r^{**}-V_0^*\|$ can be derived with a similar way and leads to the same result with \eqref{optimal performance deviation}. The proof is completed.
\end{proof}

Theorem~\ref{Thm: performance deviation} introduces the gradient-energy integral as a key quantity governing the sensitivity of performance degradation. More precisely, this integral should be interpreted as an initial-state-dependent energy functional that measures the accumulated magnitude of the nominal value function gradient along the worst-case trajectory. Moreover, the following proposition shows that this quantity is automatically finite and admits an upper bound proportional to the nominal value function.
\begin{proposition}
	Suppose that Assumptions~\ref{Assum: running cost}-\ref{Assum: error bound} hold and the robust stability criterion~\eqref{robust stability criterion} is satisfied. Then, there exists a constant $\mu>0$, such that the gradient-energy integral term in \eqref{optimal performance deviation} is bounded by
	\begin{equation}
		\int_{0}^\infty\|\nabla V_0^*\|^2 \mathrm{d}t\leq \frac{L^2}{\mu}V_0^*(x).
	\end{equation}
\end{proposition}
\begin{proof}
	As revealed by the proof of Theorem~\ref{Thm: Robust Stability}, since the robust stability criterion~\eqref{robust stability criterion} is satisfied, there exists a constant $\beta>0$ such that \eqref{time derivative of V0} holds and meanwhile $\left(c_1+\frac{1}{2}\beta c_2\right)L-\alpha_1<0$, $\frac{1}{2\beta}c_2L-\alpha_2<0$. Consider $\mu=\alpha_1-\left(c_1+\frac{1}{2}\beta c_2\right)L$. Then, along the closed-loop system trajectory \eqref{actual model} controlled by $u_0^*$ for any $r(x,u)\in\mathcal{R}$, the time derivative of the nominal optimal value function satisfies
	\begin{equation}
		\frac{\mathrm{d}V_0^*(x(t))}{\mathrm{d}t}\le-\mu\|x(t)\|^2.
	\end{equation}
	Integrating both sides over $[0,\infty)$ gives
	\begin{equation}
		-V_0^*(x)=\int_0^\infty\frac{\mathrm{d}V_0^*(x(t))}{\mathrm{d}t}\mathrm{d}t\le-\mu\int_0^\infty\|x(t)\|^2dt,
	\end{equation}
	since $\lim_{t\to\infty}x(t)=0$ is guaranteed by Theorem~\ref{Thm: Robust Stability}. Furthermore, by Lemma 1, we obtain
	\begin{equation}
		\int_0^\infty\|\nabla V_0^*(x(t))\|^2dt\le L^2\int_0^\infty\|x(t)\|^2dt\le \frac{L^2}{\mu}V_0^*(x).
	\end{equation}
	The proof is completed.
\end{proof}

\subsection{Deviation of controller}\label{4.2 controller deviation}
The analysis in \Cref{4.1 performance deviation} has established that the performance degradation can be quantified under general nonlinear dynamics and cost functional relying only on Assumptions~\ref{Assum: running cost}, \ref{Assum: value functions} and \ref{Assum: error bound}, which highlights the general applicability of the theoretical results. By contrast, quantifying the explicit deviation in the resulting controller is more involved.

In general, the optimal controller is implicitly defined through the Hamiltonian minimization in the HJB equation, which couples the controller to value function in a way that prevents a direct derivation of optimality deviation bound. To enable an explicit expression for the optimal controller and facilitate the associated deviation analysis, we introduce an additional structural assumption.
\begin{assumption}\label{Assum: control-affine}
	The nominal surrogate dynamics \eqref{surrogate model} is control-affine, i.e.,
	\begin{equation}
		\dot{x}(t)=\hat{f}(x(t),u(t))=f_0(x(t))+g_0(x(t))u(t).
	\end{equation}
	Furthermore, the running cost is of the form
	\begin{equation}\label{running cost control-affine}
		l(x,u)=q(x)+\frac{1}{2}u^\top Ru,\ R\succ0.
	\end{equation}
\end{assumption}

Assumption~\ref{Assum: control-affine} is only imposed for the analysis of optimal controller deviation in this subsection and does not affect the validity of theoretical results for performance deviation in \Cref{4.1 performance deviation}. It is worth emphasizing that the control-affine structure is imposed on the nominal system \eqref{surrogate model} rather than the actual dynamics \eqref{original system}. This reflects a common practice in surrogate modeling, where a structured model class is adopted to facilitate controller synthesis, while the actual system may exhibit more complicated nonlinear dynamics. The resulting model mismatch is captured by the approximation error term $r(x,u)$, explicitly accounted for in the robustness analysis.

We are now ready to quantify the optimality deviation in the optimal controller with the following theorem.
\begin{theorem}\label{Thm: controller deviation}
	Due to the existence of model mismatch, the nominal optimal controller $u_0^*$ given by \eqref{nominal HJB} deviates from the actual optimal controller $u^*$ corresponding to the original nonlinear system \eqref{original system}. Under Assumptions~\ref{Assum: running cost}-\ref{Assum: control-affine}, this deviation is characterized by
	\begin{equation}\label{controller deviation}
		\begin{aligned}
			\int_{0}^{\infty}(u_0^*-u^*)^\top R(u_0^*-u^*) \mathrm{d}t 
			\leq 2\Delta V_{\max}+2C_{12}\left[(V_0^*+\Delta V_{\max})\int_{0}^\infty\|\nabla V_0^*\|^2\mathrm{d}t\right]^\frac{1}{2},
		\end{aligned}
	\end{equation}
	where the integral term of $\|\nabla V_0^*\|^2$ here is evaluated along the actual optimal trajectory, i.e., \eqref{original system} controlled by $u^*$, and $C_{12},\Delta V_{\max}$ are given by Theorem~\ref{Thm: performance deviation}.
\end{theorem}
\begin{proof}
	With Assumption~\ref{Assum: control-affine}, the actual value function $V^*$ solved from the HJB equation \eqref{actual HJB} satisfies
	\begin{equation*}
		\begin{aligned}
			\min_{u}\left\{(\nabla V^*)^\top[f_0(x)+g_0(x)u+r(x,u)]+l(x,u)\right\}=0,
		\end{aligned}
	\end{equation*}
	where the running cost $l(x,u)$ is given by \eqref{running cost control-affine}. The partial derivative of $r(x,u)=f(x,u)-f_0(x)-g_0(x)u$ with respect to $u$ is $\mathcal{C}^1$ since $\frac{\partial r}{\partial u}=\frac{\partial f}{\partial u}-g_0(x)$. Denote $r_{u^*}=\frac{\partial r}{\partial u}(x,u^*)$. Accordingly, the actual optimal controller satisfies
	\begin{equation}\label{actual optimal control}
		u^*(x)=-R^{-1}\left(g_0(x)+\frac{\partial r}{\partial u}(x,u^*)\right)^\top\nabla V^*.
	\end{equation}
	With \eqref{actual HJB} and \eqref{actual optimal control}, $V^*(x)$ is the solution of
	\begin{equation}\label{actual HJB control-affine}
		\begin{aligned}
			(\nabla V^*)^\top f_0(x)+q(x)-\frac{1}{2}(\nabla V^*)^\top g_0(x)R^{-1}g_0^\top(x) \nabla V^* \\ +(\nabla V^*)^\top r(x,u^*)+\frac{1}{2}(\nabla V^*)^\top r_{u^*}R^{-1}r_{u^*}^\top\nabla V^*=0.
		\end{aligned}
	\end{equation}
	For the nominal case where $r(x,u)=0$, the HJB equation~\eqref{actual HJB control-affine} degenerates to 
	\begin{equation}\label{nominal HJB control-affine}
		(\nabla V_0^*)^\top f_0(x)+q(x)-\frac{1}{2}(\nabla V_0^*)^\top g_0(x)R^{-1}g_0^\top(x) \nabla V_0^*=0,
	\end{equation}
	which is equivalent to \eqref{nominal HJB} for control-affine case and induces the nominal optimal controller $u_0^*=-R^{-1}g_0^\top(x)\nabla V_0^*$. Then,
	\begin{equation*}
		\begin{aligned}
			(u_0^*-u^*)^\top R(u_0^*-u^*)=&(\nabla V^*)^\top r_{u^*}R^{-1}r_{u^*}^\top\nabla V^*
			+2\nabla (V^*-V_0^*)^\top g_0(x)R^{-1}r_{u^*}^\top\nabla V^*\\
			&+\nabla (V^*-V_0^*)^\top g_0(x)R^{-1}g_0^\top(x)\nabla (V^*-V_0^*)
		\end{aligned}
	\end{equation*}
	holds. With \eqref{actual HJB control-affine} and \eqref{nominal HJB control-affine}, we obtain
	\begin{equation}\label{controller deviation original formula}
		\begin{aligned}
			(u_0^*-u^*)^\top R(u_0^*-u^*)=&-2(\nabla V_0^*)^\top r(x,u^*) \\
			&-2\nabla (V^*-V_0^*)^\top\left[f_0(x)+g_0(x)u^*+r(x,u^*)\right].
		\end{aligned}
	\end{equation}
	Based on Theorem~\ref{Thm: Robust Stability} and Corollary~\ref{Coro: V^*} that confirms the closed-loop robust stability, $V^*(x(\infty))=V_0^*(x(\infty))=0$ holds. The last term in \eqref{controller deviation original formula} is the time derivative of $-(V^*-V_0^*)$ along the trajectory \eqref{original system} controlled by the actual optimal controller $u^*$. By integrating along this trajectory, the optimality deviation characterized by the nominal and actual optimal controllers satisfies 
	\begin{equation}\label{equality for controller deviation}
		\begin{aligned}
			&\int_{0}^{\infty}(u_0^*-u^*)^\top R(u_0^*-u^*) \mathrm{d}t \\
			=&-2\int_{0}^{\infty}\left[\frac{\mathrm{d}}{\mathrm{d}t}(V^*-V_0^*)+(\nabla V_0^*)^\top r\left(x(t),u^*\left(t\right)\right)\right]\mathrm{d}t \\
			=&2\left(V^*(x)-V_0^*(x)\right)-2\int_{0}^{\infty}(\nabla V_0^*)^\top r\left(x(t),u^*\left(t\right)\right)\mathrm{d}t.
		\end{aligned}
	\end{equation}
	Following similar procedures to \eqref{combine 1} and \eqref{combine 2}, we obtain
	\begin{equation}\label{inequality for controller deviation}
		\begin{aligned}
			&\left\|\int_{0}^{\infty}(\nabla V_0^*)^\top r\left(x(t),u^*\left(t\right)\right)\mathrm{d}t\right\| \\
			\leq &\left(\int_{0}^\infty\|\nabla V_0^*\|^2\mathrm{d}t\right)^{\frac{1}{2}} \left(\int_{0}^\infty(c_1\|x\|+c_2\|u^*\|)^2\mathrm{d}t\right)^{\frac{1}{2}}
			\leq C_{12}\left(V^*\int_{0}^\infty\|\nabla V_0^*\|^2\mathrm{d}t\right)^{\frac{1}{2}}.
		\end{aligned}
	\end{equation}
	With the notation $J(u,z,r)$, we obtain the relations
	\begin{equation}\label{value function relations}
		\begin{aligned}
			V_r^*(x)=J(u_0^*,x,r_0^*)&\geq V(x)=J(u_0^*,x,r), \\
			V_r^*(x)=J(u_0^*,x,r_0^*)&\geq V_0^*(x)=J(u_0^*,x,0), \\
			V(x)=J(u_0^*,x,r) &\geq V^*(x)=J(u^*,x,r),
		\end{aligned}
	\end{equation}
	which directly follow from the definitions of $V,V_0^*,V_r^*$ and the optimality properties of $u_0^*$ and $r_0^*$. Consequently, we have
	\begin{equation*}
		V^*-V_0^*=V^*-V+V-V_0^*\leq V-V_0^*\leq\Delta V_{\max}.
	\end{equation*}
	Then inequalities \eqref{inequality for controller deviation} can be further processed, i.e.,
	\begin{equation*}
		\begin{aligned}
			&\left\|\int_{0}^{\infty}(\nabla V_0^*)^\top r\left(x(t),u^*\left(t\right)\right)\mathrm{d}t\right\| 
			\leq C_{12}\left[(V_0^*+\Delta V_{\max})\int_{0}^\infty\|\nabla V_0^*\|^2\mathrm{d}t\right]^{\frac{1}{2}}.
		\end{aligned}
	\end{equation*}
	Along with \eqref{equality for controller deviation}, the optimality deviation of controller satisfies \eqref{controller deviation}. The proof is completed.
\end{proof}

\begin{remark}[An underlying analytical strategy]
	The proof of Theorem~\ref{Thm: controller deviation} also confirms that the optimality deviation between the nominal and actual optimal value functions is bounded by $V^*-V_0^*\leq \Delta V_{\max}$. Directly connecting $V_0^*(x)$ and $V^*(x)$, $u_0^*$ and $u^*$ is relatively challenging due to their implicit coupling through distinct HJB equation. Hence, the intermediate worst-case value function $V_r^*(z)$ serves as a pivotal bridge, which captures the maximal deviation induced by admissible model mismatch and allows a tractable analysis of $V^*-V_0^*$. In this sense, $V_r^*$ is not an auxiliary construct but a conceptually necessary element that supports the entire optimality deviation analysis.
\end{remark}

\section{Exemplification through the LQR lens}\label{5.linear}
To further elucidate the proposed robustness analysis, we consider its specialization to the classical LQR problem. Beyond providing a concrete and analytically transparent example, the LQR setting offers a valuable perspective from which broader implications of the proposed analysis can be more clearly understood. As will be shown, the exemplification establishes an unexpected connection between nonlinear optimality degradation and Riccati perturbation analysis \cite{AREbook1995}, revealing that the two perspectives are linked through closely related underlying structures.
\subsection{From nonlinear robustness to LQR stability}
Consider the optimal control problem of a linear system
\begin{equation}\label{actual linear system}
	\dot{x}(t)=Ax(t)+Bu(t),
\end{equation}
with the quadratic cost functional \eqref{optimal control problem} where
\begin{equation}\label{LQR: running cost}
	l(x(t),u(t))=\frac{1}{2}x^\top(t)Qx(t)+\frac{1}{2}u^\top(t)Ru(t).
\end{equation}
Suppose that the nominal surrogate model \eqref{surrogate model} is given by
\begin{equation}\label{nominal linear system}
	\dot{x}(t)=\hat{A}x(t)+\hat{B}u(t).
\end{equation}
Then, the resulting approximation error can be expressed by
\begin{equation}
	r(x(t),u(t))=\Delta Ax(t)+\Delta Bu(t),
\end{equation}
where $\Delta A=A-\hat{A},\Delta B=B-\hat{B}$. It is well-known that under such settings with $Q,R\succ 0$ and $(\hat{A},\hat{B})$ being stabilizable, the nominal optimal controller is given by $u_0^*(t)=-R^{-1}\hat{B}^\top P_0x(t)$ and the nominal optimal value function is $V_0^*(x)=\frac{1}{2}x^\top P_0x$, where $P_0$ is the unique positive definite solution of algebraic Riccati equation (ARE)
\begin{equation}\label{LQR: nominal ARE}
	P_0\hat{A}+\hat{A}^\top P_0+Q-P_0\hat{B}R^{-1}\hat{B}^\top P_0=0.
\end{equation}
Assumptions~\ref{Assum: running cost}, \ref{Assum: value functions}, \ref{Assum: error bound} and \ref{Assum: control-affine} naturally hold with $\alpha_1=\frac{1}{2}\lambda_{\min}(Q)$, $\alpha_2=\frac{1}{2}\lambda_{\min}(R)$, $c_1=\|\Delta A\|$ and $c_2=\|\Delta B\|$. In the presence of model mismatch, our objective is not to revisit the classical LQR problem, but to investigate how the optimality deviation analysis developed in Section~\ref{4.optimality} manifests itself in the classical setting.

Specializing the analysis in Theorem~\ref{Thm: Robust Stability} to the LQR setting directly yields the following result.
\begin{corollary}\label{Coro: LQR robust stability}
	Suppose that $(\hat{A},\hat{B})$ is stabilizable, $Q,R\succ0$. Then, the nominal optimal controller $u_0^*=-R^{-1}\hat{B}^\top P_0$ in the LQR setting stabilizes the actual linear system \eqref{actual linear system} if
	\begin{equation}\label{LQR: robust stability criterion}
		\frac{\|P_0\|}{\lambda_{\min}(Q)}\left[2\|\Delta A\|+\frac{\|P_0\|}{ \lambda_{\min}(R)}\|\Delta B\|^2\right]<1.
	\end{equation}
\end{corollary}
\begin{proof}
	Since the nominal optimal value function satisfies $V_0^*=\frac{1}{2}x^\top P_0 x$, the Lipschitz constant in Lemma~\ref{Lem: gradient of value function} can be selected as $L=\|P_0\|$. Substituting the quantities $c_1,c_2,\alpha_1, \alpha_2$ with $\|\Delta A\|,\|\Delta B\|,\frac{1}{2}\lambda_{\min}(Q), \frac{1}{2}\lambda_{\min}(R)$ respectively in \eqref{robust stability criterion} yields the condition \eqref{LQR: robust stability criterion} for closed-loop robust stability. The desired result follows immediately from Theorem~\ref{Thm: Robust Stability}.
\end{proof}

\Cref{Coro: LQR robust stability} provides an explicit robust stability criterion for perturbed LQR closed-loop systems, whose structure closely parallels classical Lyapunov-based robustness conditions in existing literature. Taking \cite{Patel1977robustness} for example, the condition for closed-loop stability under linear perturbations is provided by
\begin{equation}\label{LQR: existing robust stability}
	\|\Delta A\|+\|\Delta B\|\|K_0\|\leq \frac{\lambda_{\min}(Q+P_0\hat{B}R^{-1} \hat{B}^\top P_0)}{2\|P_0\|},
\end{equation}
where $K_0=-R^{-1}\hat{B}^\top P_0$. To make the intrinsic consistency of \eqref{LQR: existing robust stability} and Corollary~\ref{Coro: LQR robust stability} more transparent, we illustrate with two special cases.

\textit{Case 1:} $\|\Delta B\|=0$. Under this circumstance, conditions \eqref{LQR: robust stability criterion} and \eqref{LQR: existing robust stability} degenerate to $\|\Delta A\|< \frac{\lambda_{\min}(Q)}{2\|P_0\|}$ and $\|\Delta A\|\leq \frac{\lambda_{\min}(Q+P_0\hat{B}R^{-1} \hat{B}^\top P_0)}{2\|P_0\|}$ respectively. Note that as $K_0^\top RK_0=P_0\hat{B}R^{-1} \hat{B}^\top P_0$, $\lambda_{\min}(Q) \leq\lambda_{\min}(Q+K_0^\top RK_0)$ holds. The slight conservatism in the derived bound represents an inevitable technical price for deploying a unified and general nonlinear robustness analysis framework. Nevertheless, both boundaries match identically in terms of algebraic scaling and order of magnitude.

\textit{Case 2:} $\|\Delta A\|=0$. Under this circumstance, the condition \eqref{LQR: robust stability criterion} degenerates to $\|\Delta B\|<\frac{\sqrt{ \lambda_{\min}(Q)\lambda_{\min}(R)}}{\|P_0\|}$, which satisfies
\begin{subequations}
	\begin{align}
		\|\Delta B\|&<\frac{1}{\|P_0\|}\sqrt{\lambda_{\min}(Q)\lambda_{\min}(R)}
		\leq \frac{1}{\|P_0\|}\left[\frac{\lambda_{\min}(Q)}{2\|K_0\|}+ \frac{\lambda_{\min}(R)}{2\|K_0\|}\|K_0\|^2\right] \\
		&\leq \frac{\lambda_{\min}(Q)+\lambda_{\max}(K_0^\top RK_0)}{2\|P_0\|\|K_0\| }. \label{LQR: connection for robust stability}
	\end{align}
\end{subequations}
Meanwhile, the condition \eqref{LQR: existing robust stability} degenerates to $\|\Delta B\|\leq \frac{\lambda_{\min}(Q+K_0^\top RK_0)}{2\|P_0\|\|K_0\|}$ when $\|\Delta A\|=0$, which also satisfies \eqref{LQR: connection for robust stability}. Therefore, both conditions admit the same upper bound in terms of the Riccati solution and the nominal feedback gain, demonstrating that the proposed criterion preserves the essential robustness structure of existing LQR analyses even under isolated input perturbations.

Together, the above two limiting cases demonstrate that the proposed nonlinear robust stability criterion admits a natural specialization to the classical LQR setting, and remains structurally consistent with existing robustness conditions despite being derived from a unified nonlinear analysis framework.

\subsection{From performance deviation to Riccati perturbation}
To specialize the performance deviation analysis to the LQR setting, the central difficulty lies in characterizing the gradient-energy integral term $\int_{0}^\infty\|\nabla V_0^*\|^2\mathrm{d}t$ evaluated along the worst-case trajectory \eqref{nominal worst trajectory u0} appearing in Theorem~\ref{Thm: performance deviation}. By leveraging the quadratic structure of the nominal value function, this quantity admits a Lyapunov characterization, which finally results in a perturbation bound closely related to classical ARE sensitivity analysis.
\begin{theorem}\label{Thm: LQR perturbation}
	Suppose that $(\hat{A},\hat{B})$ is stabilizable, $Q,R\succ0$ and the condition \eqref{LQR: robust stability criterion} is satisfied. Then, the actual optimal value function corresponding to the system \eqref{actual linear system} is given by $V^*(x)=\frac{1}{2}x^\top Px$, where $P$ is the unique and positive definite solution of ARE
	\begin{equation}\label{LQR: actual ARE}
		PA+A^\top P+Q-PBR^{-1}B^\top P=0.
	\end{equation}
	Further, the matrix perturbation deviation satisfies $\forall\varepsilon>0$,
	\begin{equation}\label{LQR: perturbation bound}
		P-P_0\leq C_{12}\left[\frac{1}{2\varepsilon(1-\beta)}S+2\varepsilon P_0\right]+\frac{C_{12}^2}{(1-\beta)}S +\frac{\varepsilon C_{12}^3}{2(1-\beta)}S,
	\end{equation}
	where $\beta=\frac{2\left(c_1+c_2\|K_0\|\right) \lambda_{\max}(S)}{\lambda_{\min}(P_0^2)}$, $K_0=-R^{-1}\hat{B}^\top P_0$, $\hat{A}_c=\hat{A}+\hat{B}K_0$, and $S$ is the unique solution of Lyapunov equation
	\begin{equation}\label{LQR: Lyapunov equation}
		\hat{A}_c^\top S+S\hat{A}_c+P_0^2=0.
	\end{equation}
\end{theorem}
\begin{proof}
	Under the condition \eqref{LQR: robust stability criterion}, the actual linear system $(A,B)$ is stabilizable according to Corollary~\ref{Coro: LQR robust stability}. Hence, the associated ARE \eqref{LQR: actual ARE} admits a unique positive definite solution \cite{anderson2007optimal}, i.e., $P\succ0$.
	
	As the nominal optimal value function admits a quadratic form $V_0^*(x)=\frac{1}{2}x^\top P_0x$, the gradient-energy term in \eqref{optimal performance deviation} is rewritten as $\int_{0}^\infty\|\nabla V_0^*\|^2\mathrm{d}t=\int_{0}^\infty x^\top(t)P_0^2x(t)\mathrm{d}t$. To relate this key quantity to a Lyapunov characterization of the nominal closed-loop system, we calculate the time derivative of $x^\top Sx$ along trajectory \eqref{nominal worst trajectory u0}. Since $(\hat{A},\hat{B})$ is stabilizable and $Q,R\succ0$, the nominal closed-loop matrix $\hat{A}_c=\hat{A}+\hat{B}K_0$ is Hurwitz, and thus Lyapunov equation \eqref{LQR: Lyapunov equation} admits a unique positive definite solution \(S\succ0\). Then, the time derivative of $x^\top Sx$ satisfies
	\begin{equation}\label{LQR: time derivative of xTSx}
		\begin{aligned}
			\frac{\mathrm{d}}{\mathrm{d}t}(x^\top Sx)&=2x^\top(t)S\left[\hat{A}_cx(t)+r_0^*(x(t), u_0^*(t))\right]\\&=x^\top(t)(S\hat{A}_c+\hat{A}_c^\top S)x(t)+2x^\top(t)Sr_0^*=-x^\top(t)P_0^2x(t)+2x^\top(t)Sr_0^*.
		\end{aligned}
	\end{equation}
	With \eqref{nominal worst-case error}, we obtain
	\begin{equation}
		\|r_0^*(x,u_0^*)\|=c_1\|x\|+c_2\|u_0^*\|\leq \left(c_1+c_2\left\|K_0\right\|\right)\|x\|.
	\end{equation}
	Combining this uncertainty bound with \eqref{LQR: time derivative of xTSx} yields
	\begin{equation*}
		\begin{aligned}
			\int_{0}^\infty x^\top(t)P_0^2x(t)\mathrm{d}t&=\int_{0}^\infty -\frac{\mathrm{d}}{\mathrm{d}t}(x^\top Sx)+2x^\top(t)Sr_0^*\mathrm{d}t\\
			&=x^\top Sx+\int_{0}^\infty2\lambda_{\max}(S)(c_1+c_2\|K_0\|)\|x(t)\|^2\mathrm{d}t\\
			&\leq x^\top Sx+\int_{0}^\infty \frac{2\lambda_{\max}(S)(c_1+c_2\|K_0\|)}{ \lambda_{\min}(P_0^2)}x^\top(t)P_0^2x(t)\mathrm{d}t.
		\end{aligned}
	\end{equation*}
	Rearranging the above inequality gives
	\begin{equation}\label{LQR: upper bound for integral term}
		\int_{0}^\infty\|\nabla V_0^*\|^2\mathrm{d}t=\int_{0}^\infty x^\top(t)P_0^2x(t)\mathrm{d}t \leq \frac{1}{1-\beta}x^\top Sx.
	\end{equation}
	The above quadratic characterization allows the performance deviation bound in Theorem~\ref{Thm: performance deviation} to be expressed entirely in terms of matrix quantities. Substituting \eqref{LQR: upper bound for integral term} into \eqref{optimal performance deviation} in Theorem~\ref{Thm: performance deviation} yields that $\forall\varepsilon>0$,
	\begin{equation*}
		\begin{aligned}
			\Delta V_{\max}&\leq\frac{C_{12}}{2}\left(\frac{1}{1-\beta} x^\top Sx\right)^{\frac{1}{2}}\sqrt{x^\top \left(\frac{C_{12}^2}{1-\beta}S +4P_0\right)x}+\frac{C_{12}^2}{2(1-\beta)}x^\top Sx\\
			&\leq \frac{C_{12}}{4\varepsilon}\frac{1}{1-\beta} x^\top Sx +\frac{\varepsilon C_{12}}{4}x^\top \left(\frac{C_{12}^2}{1-\beta}S +4P_0\right)x +\frac{C_{12}^2}{2(1-\beta)}x^\top Sx,
		\end{aligned}
	\end{equation*}
	where the Young's inequality is applied in the last step. With \eqref{value function relations}, we have
	\begin{equation}
		V^*-V_0^*=\frac{1}{2}x^\top(P-P_0)x\leq \Delta V_{\max},
	\end{equation}
	and the performance deviation characterization in \Cref{4.optimality} finally induces the perturbation bound \eqref{LQR: perturbation bound}. The proof is completed.
\end{proof}

The perturbation bound \eqref{LQR: perturbation bound} in Theorem~\ref{Thm: LQR perturbation} exhibits an intrinsic consistency with classical Riccati sensitivity results \cite{AREbook1995,konstantinov1993perturbation,kenney1988sensitivity}. In particular, focusing on the first-order asymptotic characteristics where $C_{12}=\max\left\{\sqrt{\frac{2}{\alpha_1}}\|\Delta A\|, \sqrt{\frac{2}{\alpha_2}}\|\Delta B\|\right\}\to0$, we can take the norm from both sides of \eqref{LQR: perturbation bound} and obtain
\begin{equation}
	\|P-P_0\|\leq C_{12}\left[\frac{1}{2\varepsilon(1-\beta)}\|S\|+2\varepsilon \|P_0\|\right] +O(C_{12}^2).
\end{equation}
Since the above inequality holds for all $\varepsilon>0$, we consider $\varepsilon^* =\sqrt{\|S\|/4\varepsilon(1-\beta)\|P_0\|}$ to minimize the first-order perturbation upper bound, yielding
\begin{equation}\label{LQR: first-order perturbation}
	\|P-P_0\|\leq 2C_{12}\sqrt{\frac{1}{(1-\beta)}\|S\|\|P_0\|}+O(C_{12}^2),
\end{equation}
which is structurally consistent with classical results via ARE sensitivity analysis \cite[Theorem~2.4]{kenney1988sensitivity}.

Despite this alignment in the LQR problem, the underlying viewpoints remain fundamentally different. Classical sensitivity methods typically initiate from explicit perturbations of the system matrices, and directly study the resulting variation of Riccati solution. In contrast, the proposed framework is rooted in the optimality deviation analysis for general nonlinear systems, and arrives at a Riccati perturbation characterization as a consequence. This observation reveals that although developed from different starting points, nonlinear optimality robustness and classical Riccati perturbation theory are connected through closely related analytical structures.

\subsection{Illustrative examples}
To verify the LQR-specialized analysis, we present a numerical case study as a benchmark. Consider the linear system \eqref{actual linear system} with
\begin{equation}
	A=\begin{bmatrix}
		0.8& -0.9\\ 1.0& 0.8
	\end{bmatrix}, B=\begin{bmatrix}
		1& 0\\0& 1
	\end{bmatrix}.
\end{equation}
The running cost $l(x,u)$ satisfies \eqref{LQR: running cost} where $Q=I_2, R=1$. To emulate an imperfect surrogate modeling, we construct a family of nominal models by introducing perturbations in both the state and input matrices, i.e.,
\begin{equation}
	\hat{A}=A+\delta\begin{pmatrix}
		0.2& -0.3\\ 0.1& 0.1
	\end{pmatrix},\hat{B}=B-\delta\begin{pmatrix}
		0.4& 0\\0& 0.2
	\end{pmatrix},
\end{equation}
which naturally satisfies Assumption~\ref{Assum: error bound}. By varying $\delta$, different levels of modeling mismatch can be systematically generated.
As $\delta$ varies from $0$ to $0.2$, the nominal system $(\hat{A},\hat{B})$ remains controllable and thus stabilizable, since the matrix $[\hat{B}\ \hat{A}\hat{B}]$ is full row rank for all considered values of $\delta$. Therefore, the corresponding Riccati solution $P_0(\delta)$, feedback gain $K_0(\delta)$, and auxiliary matrix $S(\delta)$ are well-defined and can be evaluated along the entire perturbation path.

The Riccati solution $P$ for the actual optimal value function $V^*$ is computed with the true linear system $(A,B)$. For each $\delta$, the resulting perturbation magnitude can then be quantified via $\|P-P_0\|$. Meanwhile, the theoretical upper bound provided by Theorem~\ref{Thm: LQR perturbation} is evaluated using the corresponding quantities $S$, $\beta$, and $C_{12}$.

The quantitative evaluation of the derived perturbation results is presented in \Cref{fig:LQR-comparison}. As the level of model mismatch $\delta$ increases, the nominal system progressively deviates from the actual dynamics, leading to a larger optimality deviation. As shown in \Cref{fig:LQR-comparison}, the actual performance deviation characterized by $\|P-P_0\|$ remains strictly enclosed by the theoretical upper bound~\eqref{LQR: perturbation bound} established in Theorem~\ref{Thm: LQR perturbation} throughout the entire uncertainty range. Additionally, the first-order perturbation characterization in \eqref{LQR: first-order perturbation} closely follows the evolution of actual deviation and provides a less conservative estimate of the Riccati sensitivity.
\begin{figure}[htbp]
	\centering
	\includegraphics[width=0.72\textwidth]{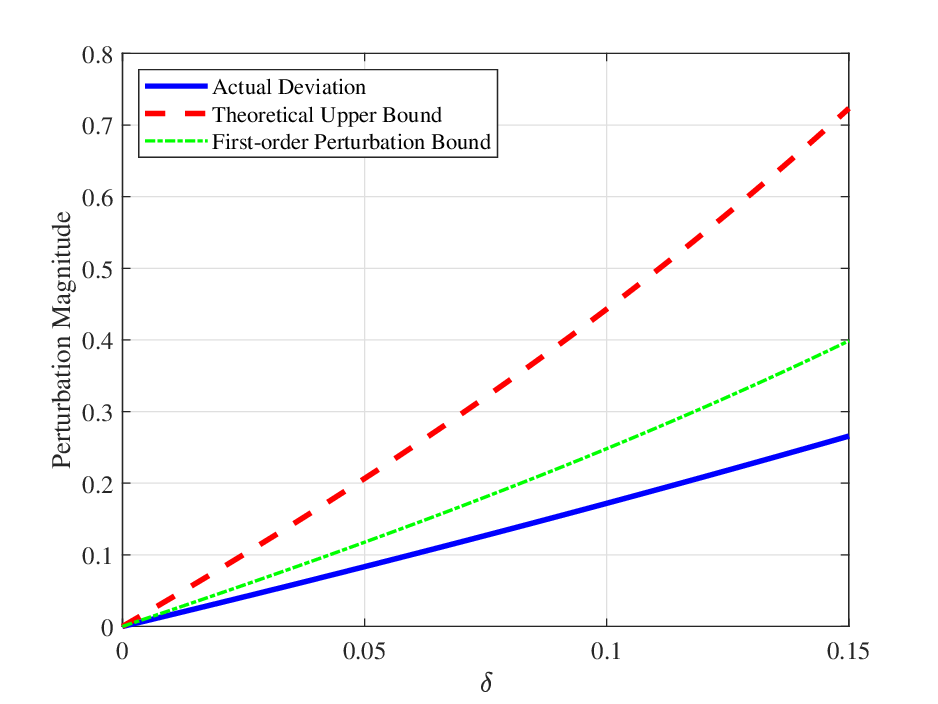}
	\caption{Numerical verification of \Cref{Thm: LQR perturbation}: Comparison among the actual deviation of Riccati solution $\|P-P_0\|$, the first-order perturbation bound \eqref{LQR: first-order perturbation}, and the theoretical upper bound \eqref{LQR: perturbation bound} as the model mismatch level $\delta$ increases.}
	\label{fig:LQR-comparison}
\end{figure}
\begin{figure}[htbp]
	\centering
	\includegraphics[width=0.72\textwidth]{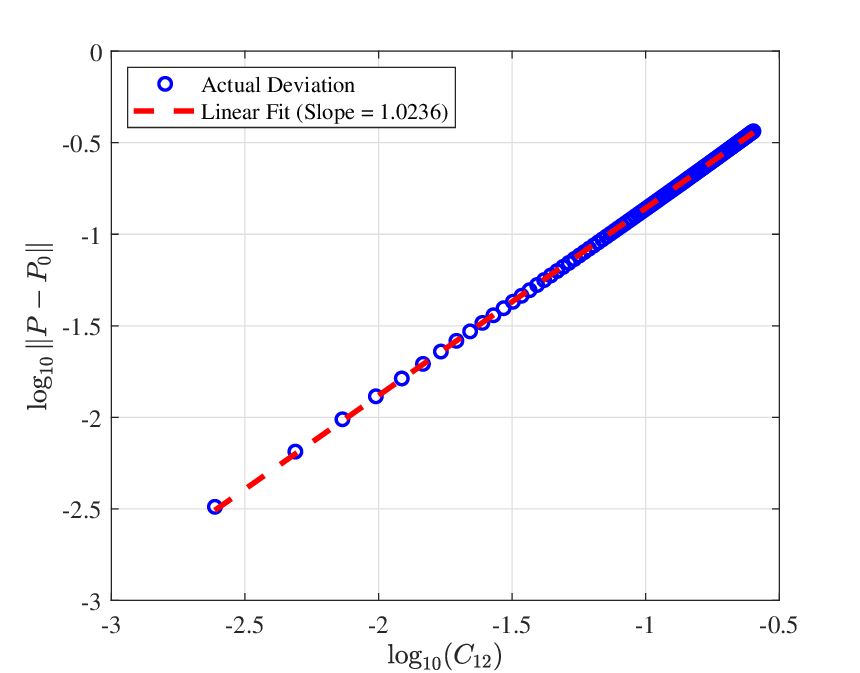}
	\caption{Logarithmic relationship between the actual Riccati perturbation $\|P-P_0\|$	and the uncertainty coefficient $C_{12}$. The fitted slope close to unity confirms the first-order perturbation law $\|P-P_0\|=O(C_{12})$.}
	\label{fig:LQR-first order}
\end{figure}

To examine the asymptotic behavior predicted by Theorem~\ref{Thm: LQR perturbation} in a more transparent way, \Cref{fig:LQR-first order} depicts the logarithmic relationship between $\|P-P_0\|$ and the uncertainty coefficient $C_{12}=\max\left\{\sqrt{\frac{2}{\alpha_1}}\|\Delta A\|,\sqrt{\frac{2}{\alpha_1}} \|\Delta B\|\right\}$. A linear fitting yields a slope of approximately $1.0236$, which quantitatively confirms the first-order perturbation relation $\|P-P_0\|=O(C_{12})$ given by \eqref{LQR: first-order perturbation}. This observation indicates that the proposed nonlinear optimality robustness framework not only provides a rigorous perturbation bound for the LQR setting, but also accurately captures the intrinsic first-order sensitivity structure of the ARE solution with respect to model mismatch.

\section{Computational evaluation of optimality deviation}\label{6.computation}
Section~\ref{4.optimality} has established quantitative characterizations of optimality deviation in terms of both value function and controller. To complement these theoretical results, this section focuses on their computational realization. First, we present a unified approach for computing diverse forms of optimality deviation. Numerical simulations are then conducted to validate the proposed robustness analysis.

\subsection{Unified computation across variants}
Recall that Theorems~\ref{Thm: performance deviation} and~\ref{Thm: controller deviation} explicitly characterize the optimality deviations in both performance and controller through the nominal optimal value function $V_0^*$ and the error bound coefficients $(c_1,c_2)$ for model mismatch. Despite these theoretical characterizations, a closer inspection of \eqref{optimal performance deviation} and \eqref{controller deviation} reveals a computational bottleneck, which is the main focus of this subsection.

Taking the performance deviation for instance, Theorem~\ref{Thm: performance deviation} and its proof have provided two forms of the deviation bound, \eqref{optimal performance deviation} and \eqref{optimal performance deviation version 2}, both of which necessitates the computation of worst-case value function $V_r^*$. Moreover, an alternative form of performance deviation can be derived.
\begin{corollary}\label{Coro: performance deviation}
	Due to the model mismatch, applying $u_0^*$ to the actual nonlinear system \eqref{original system} leads to an extra cost in \eqref{optimal control problem}. With Assumptions~\ref{Assum: running cost}-\ref{Assum: error bound}, this extra cost satisfies
	\begin{equation}\label{optimal performance deviation version 3}
		\|V-V_0^*\|\leq C_{12}\left(V_0^*\int_{0}^\infty\|\nabla V_r^*\|^2\mathrm{d}t\right)^{ \frac{1}{2}},
	\end{equation}
	where $V_r^*$ is solved by \eqref{actual HJI}, and the integral term of $\|\nabla V_r^*\|^2$ is evaluated along nominal optimal trajectory, i.e.,
	\begin{equation}\label{nominal optimal trajectory}
		\dot{x}(t)=\hat{f}(x(t),u_0^*(t)),\ x(0)=x.
	\end{equation}
\end{corollary}
\begin{proof}
	Following similar procedures to \eqref{nominal worst-case maximum}-\eqref{actual HJI} in the proof of Theorem~\ref{Thm: performance deviation}, we can calculate the time derivative of $V_r^*-V_0^*$ along trajectory \eqref{nominal optimal trajectory}, satisfying
	\begin{equation*}
		\frac{\mathrm{d}}{\mathrm{d}t}(V_r^*-V_0^*)=-(c_1\|x\|+c_2\|u_0^*\|)\|\nabla V_r^*\|,
	\end{equation*}
	which is obtained by combining \eqref{nominal HJB} and \eqref{actual HJI}. The upper bound for extra cost \eqref{optimal performance deviation version 3} can be obtained following \eqref{combine 1} and \eqref{combine 2}, then the proof is completed.
\end{proof}

Up to this point, three different characterizations of performance deviation have been established. These variants provide complementary ways to evaluate the optimality deviation under different computational requirements. In particular, once the nominal optimal value function $V_0^*$ is available, one may first solve $V_r^*$ and directly evaluate the deviation through \eqref{optimal performance deviation version 3}. Alternatively, the worst-case trajectory \eqref{nominal worst trajectory u0} can be constructed based on $V_r^*$, allowing the performance deviation bound to be computed with \eqref{optimal performance deviation} or \eqref{optimal performance deviation version 2}.

In practical situations where explicit trajectory computation is inconvenient, the integral term $\int_{0}^\infty\|\nabla V_0^*\|^2\mathrm{d}t$ in \eqref{optimal performance deviation} and \eqref{optimal performance deviation version 2} can also be approached, which further provides a tractable computation of the optimality deviation. Specifically, denote
\begin{equation}
	\begin{aligned}
		J_{\text{grad}}(x,u_0^*,r)=\int_{0}^\infty\|\nabla V_0^*\|^2\mathrm{d}t,& \\
		\dot{x}(t)=\hat{f}(x(t),u_0^*(t))+r(x(t),u_0^*(t)),&\ x(0)=x.
	\end{aligned}
\end{equation}
Then, $J_{\text{grad}}$ can be regarded as a cost functional in an optimization problem, whose maximum satisfies the HJI equation
\begin{equation}
	\max_{r\in\mathcal{R}}\left\{(\nabla V_{\text{grad}})^\top\left[\hat{f}(x,u_0^*)+r(x,u_0^*) \right]+\|\nabla V_0^*\|^2\right\}=0,
\end{equation}
equivalently written as
\begin{equation}\label{compute maximum integral V_0 gradient}
	(\nabla V_{\text{grad}})^\top\hat{f}(x,u_0^*)+\|\nabla V_0^*\|^2=-(c_1\|x\|+c_2\|u_0^*\|) \|\nabla V_{\text{grad}}\|.
\end{equation}
Note that \eqref{actual HJI} and \eqref{compute maximum integral V_0 gradient} shares a common structure of partial differential equation (PDE). Thus, the performance deviation with the form \eqref{optimal performance deviation} can be computationally characterized once $V_0^*$ is available. In this way, the explicit computation of optimality deviation leads to a unified computational formulation despite different variants. To solve the PDE \eqref{actual HJI} (or \eqref{compute maximum integral V_0 gradient}), an iterative scheme is developed in \Cref{Algorithm 1}.

\begin{algorithm}[H]
	\caption{Iterative Computation of $V_r^*$}\label{Algorithm 1}
	\begin{algorithmic}
		\STATE \textbf{Initialization}
		\STATE \hspace{0.5cm}{Set $k=0$, choose an initial guess $ V_r^{(0)} $ and a sufficiently small $\varepsilon>0$.} 
		\STATE \textbf{Worst-case Error Update}
		\STATE \hspace{0.5cm}{Compute the worst-case approximation error based on the current value function $V_r^{(k)}$, i.e.,
			\begin{equation}\label{Policy update: VF}
				r^{(k)}(x,u_0^*)=(c_1\|x\|+c_2\|u_0^*\|)\dfrac{\nabla V_r^{(k)}}{\left\|\nabla V_r^{(k)} \right\|}.
		\end{equation}}
		\STATE \textbf{Value Function Update}
		\STATE \hspace{0.5cm}{Update the value function $V_r^{(k+1)}$ by solving the PDE
			\begin{equation}\label{Policy evaluation: VF}
				\left(\nabla V_r^{(k+1)}\right)^\top\left[\hat{f}(x,u_0^*)+r^{(k)}(x,u_0^*)\right]= -l(x,u_0^*).
		\end{equation}}
		\STATE \textbf{Convergence Check}
		\STATE \hspace{0.5cm}{If $\|V_{r}^{(k+1)}-V_{r}^{(k)}\|<\epsilon$, stop iteration and return $V_{r}^{(k+1)}$. Otherwise, set $k=k+1$ and continue the iteration.}
	\end{algorithmic}
	\label{alg1}
\end{algorithm}

The convergence of \Cref{Algorithm 1} is ensured by the following theorem.
\begin{theorem}\label{Thm: Convergence}
	Let Assumptions~\ref{Assum: running cost}-\ref{Assum: error bound} hold and the robust stability criterion \eqref{robust stability criterion} be satisfied. Then, the sequence $\{V_r^{(k)}\}$ generated by Algorithm~1 converges to the solution $V_r^*$ of~\eqref{actual HJI}.
\end{theorem}
\begin{proof}
	Denote $r^{(k)}(x,u_0^*)=\rho(x,u_0^*)w^{(k)}$, $\rho(x,u)=c_1\|x\|+c_2\|u\|$ and $w^{(k)}=\frac{\nabla V_r^{(k)}}{\|\nabla V_r^{(k)}\|}$. With \eqref{Policy evaluation: VF}, we obtain
	\begin{equation}
		\begin{aligned}
			\nabla\left(V_r^{(k+1)}-V_r^{(k)}\right)^\top\left[\hat{f}(x,u_0^*)+\rho(x,u_0^*)w^{(k)}\right] \\ =-(\nabla V_r^{(k)})^\top\rho(x,u_0^*)\left(w^{(k)}-w^{(k-1)}\right), k\geq1.
		\end{aligned}
	\end{equation}
	Integrating both sides of \eqref{Policy evaluation: VF} along the trajectory
	\begin{equation*}
		\dot{x}=\hat{f}(x,u_0^*)+\rho(x,u_0^*)w^{(k)},\ x(0)=x,
	\end{equation*}
	we have
	\begin{equation}\label{proof of Theorem 4: integration}
		\begin{aligned}
			V_r^{(k+1)}(x)-V_r^{(k)}(x)=\int_{0}^{\infty}(\nabla V_r^{(k)})^\top\rho(x,u_0^*) \left(w^{(k)}-w^{(k-1)}\right)\mathrm{d}t,
		\end{aligned}
	\end{equation}
	since $\rho(x,u)w^{(k)}\in\mathcal{R}$ and then the closed-loop stability is ensured by Theorem~\ref{Thm: Robust Stability}. Note that $w^{(k)}$ updated at each iteration satisfies
	\begin{equation}
		w^{(k)}=\arg\max_{\|w\|\leq 1} (\nabla V_r^{(k)})^\top\rho(x,u_0^*)w,
	\end{equation}
	rendering the right hand side of \eqref{proof of Theorem 4: integration} positive, $\forall x\in \mathbb{X}$, and $V_r^{(k+1)}(x)\geq V_r^{(k)}(x)$. Consequently, the sequence $\{V_r^{(k)}\}$ is monotonically increasing. Meanwhile, the worst-case approximation error at each iteration satisfies $r^{(k)}\in\mathcal{R}$. By Theorem~\ref{Thm: Robust Stability}, the nominal optimal controller $u_0^*$ robustly stabilizes the corresponding closed-loop system under \eqref{robust stability criterion}. As a result, the associated value function remains finite for every iteration, implying that the sequence $\{V_r^{(k)}\}$ is upper-bounded. Therefore, the monotone sequence $\{V_r^{(k)}\}$ converges to a limit function $\bar V_r$. Passing to the limit in \eqref{Policy evaluation: VF} shows that $\bar V_r$ satisfies \eqref{actual HJI}, i.e., $\bar V_r=V_r^*$. The proof is completed.
\end{proof}

Since the linear PDE \eqref{Policy evaluation: VF} for value function update generally admits no analytical solution, a numerical approximation is required at each iteration of \Cref{Algorithm 1}. Among numerous numerical approaches for solving PDEs, we adopt the classical Galerkin projection method \cite{BEARD19972159} owing to its simplicity and compatibility with the proposed iterative computational scheme. Let $\{ \varphi_j(x) \}_{j=1}^N$ be a set of linearly independent basis functions, and the value function $V_r^{(k+1)}(x)$ updated at the $k$-th iteration is approximated as
\begin{equation}
	\hat{V}_r^{(k+1)}(x)=\sum_{j=1}^N\theta_j^{(k+1)}\varphi_j(x)=\left(\theta^{(k+1)}\right)^\top \varphi(x).
\end{equation}
At each value function update step, the Galerkin approximation is constructed for the linear PDE parameterized by the current worst-case error $r^{(k)}$. Substituting this approximation into \eqref{Policy evaluation: VF} yields the residual
\begin{equation}
	R(x)=\left(\theta^{(k+1)}\right)^\top\nabla\varphi(x)\left[\hat{f}(x,u_0^*)+r^{(k)}(x,u_0^*) \right]+l(x,u_0^*).
\end{equation}
By requiring the residual to be orthogonal to the basis functions over the state space $\mathbb{X}$, i.e.,
\begin{equation*}
	\left\langle R(x),\varphi_j(x)\right\rangle_\mathbb{X}=\int_{\mathbb{X}}\varphi_j(x)R(x) \mathrm{d}x=0,\ \forall j=1,\cdots,N,
\end{equation*}
we obtain a set of linear equations
\begin{equation}
	A^{(k)}\theta^{(k+1)}=b,
\end{equation}
where the elements of square matrix $A^{(k)}\in\mathbb{R}^{N\times N}$ and vector $b\in\mathbb{R}^N$ satisfy
\begin{equation}
	\begin{aligned}
		A^{(k)}_{ij}&=\int_{\mathbb{X}}\varphi_i(x)\left[(\nabla\varphi_j)^\top \overline{f}^{(k)}(x) \right]\mathrm{d}x, \\ \overline{f}^{(k)}&=\hat{f}+r^{(k)}, \ b_i=-\int_{\mathbb{X}}\varphi_i(x)l(x,u_0^*)\mathrm{d}x.
	\end{aligned}
\end{equation}
Note that the projection matrix $A^{(k)}$ is updated according to the current worst-case error $r^{(k)}$. Since the Galerkin projection itself is a standard computational technique, its convergence to the exact solution $\hat{V}_r^{(k)}\rightarrow V_r^{(k)}$ as $N \rightarrow \infty$ has been well established in the numerical analysis literature, e.g., \cite{BEARD19972159}, and hence is omitted here for brevity.

\subsection{Simulation results}
To validate the proposed robustness analysis, we consider a nonlinear control-affine system governed by
\begin{equation}
	\begin{bmatrix} \dot{x}_1 \\ \dot{x}_2 \end{bmatrix} = \begin{bmatrix} -x_1+x_2 \\ -\frac{1}{2}(x_1+x_2) +\frac{1}{2}x_1^2x_2 \end{bmatrix} + \begin{bmatrix} 0 \\ x_1 \end{bmatrix}u.
\end{equation}
The running cost is defined as $l(x,u)=\frac{1}{2}(x^\top x+u^\top u)$. For this benchmark problem, the actual optimal value function given by \eqref{actual HJB} can be analytically verified as $V^*(x) = \frac{1}{4}x_1^2 + \frac{1}{2}x_2^2$, which yields the actual optimal controller $u^*(x) = -x_1x_2$ and serves as the a priori ground truth for evaluation of optimality deviation.

To emulate a practical surrogate modeling situation, we consider a scenario where the dominant nonlinear structures are successfully captured, while the corresponding coefficients contain estimation errors. Accordingly, we employ a family of surrogate models where both the drift dynamics and the control effectiveness are intentionally perturbed, leading to a nominal model of the form
\begin{equation*}
	\begin{bmatrix} \dot{\hat{x}}_1 \\ \dot{\hat{x}}_2 \end{bmatrix} = \begin{bmatrix} -x_1+x_2 \\ -\frac{1}{2}(x_1+x_2)+\frac{1}{2}(1-\delta_1) x_1^2x_2 \end{bmatrix} + \begin{bmatrix} 0 \\ (1-\delta_2)x_1 \end{bmatrix}u.
\end{equation*}
The resulting approximation error $r(x,u)=[0\ \frac{1}{2}\delta_1x_1^2x_2 +\delta_2x_1u]^\top$ naturally satisfies Assumption~\ref{Assum: error bound} on a compact operating region $\mathbb{X}=[-M,M]^2$, from which the corresponding uncertainty bounds can be explicitly estimated as $c_1 = \frac{1}{2}|\delta_1|M^2$ and $c_2 = |\delta_2| M$. By adjusting the parameters $\delta_1,\delta_2$, different levels of state- and input-dependent modeling inaccuracies can be generated, enabling a systematic investigation of their influence on closed-loop stability and optimality robustness.

For each surrogate model, the nominal optimal value function $V_0^*$ and the associated feedback controller $u_0^*$ are obtained by solving \eqref{nominal HJB} via the Galerkin projection, where polynomial basis functions up to the 4th order are employed. Subsequently, \Cref{Algorithm 1} is implemented to compute the worst-case value function $V_r^*$. Under the setting $M=1.5, \delta_1=0.15, \delta_2=0.2$, the iterative convergence behavior of \Cref{Algorithm 1} is illustrated in \Cref{fig:convergence}. As guaranteed by \Cref{Thm: Convergence}, the Galerkin projection coefficients $\theta_j$ converge rapidly within a few iterations, yielding a numerically efficient solution of underlying PDE~\eqref{actual HJI}.

\begin{figure}[htbp]
	\centering
	\includegraphics[width=0.7\linewidth]{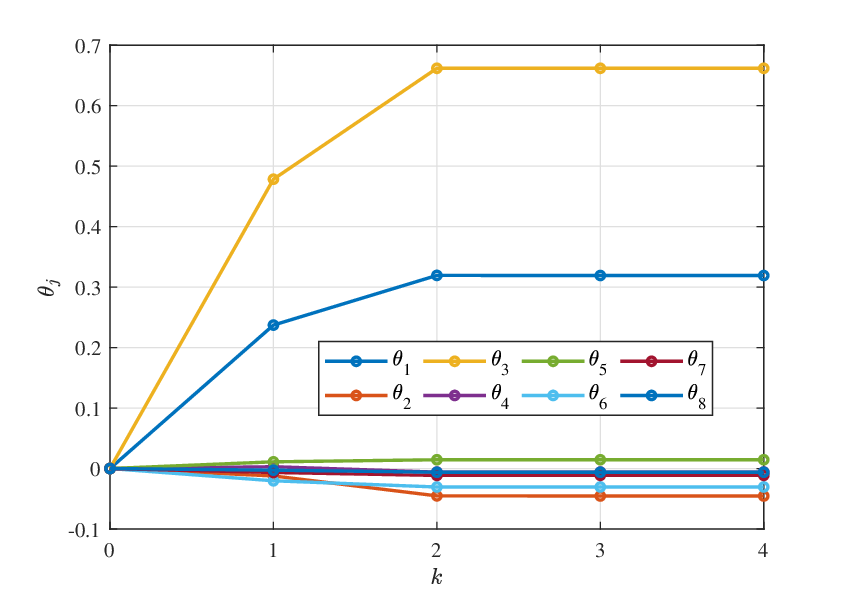}
	\caption{Convergence behavior of the Galerkin expansion coefficients $\theta_j$ via \Cref{Algorithm 1} starting from a zero initialization ($k=0$).}
	\label{fig:convergence}
\end{figure}

\begin{figure}[htbp]
	\centering
	\includegraphics[width=\textwidth]{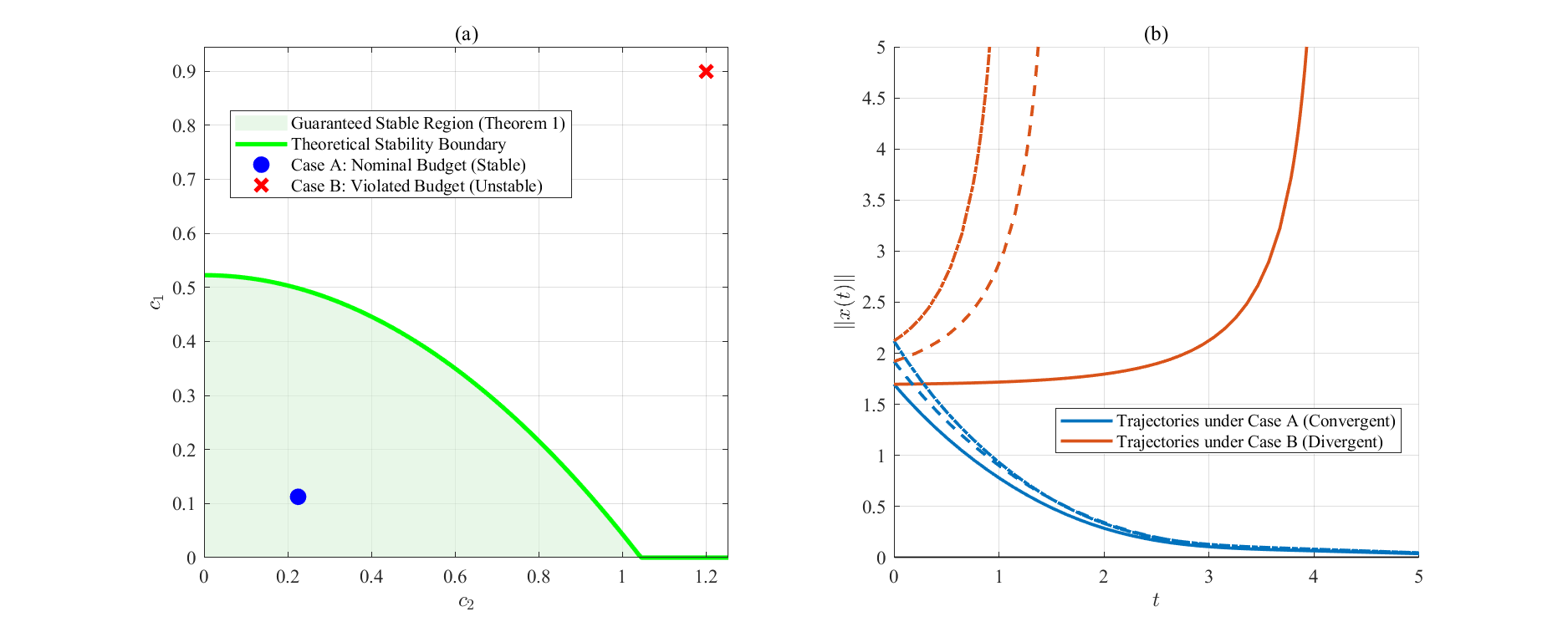}
	\caption{Joint macro-micro verification of robust stability under model mismatch. (a) Theoretical robust stability budget region for the uncertainty bounds $c_1$ and $c_2$ based on Theorem~\ref{Thm: Robust Stability}. (b) State norm of trajectories under Case A (convergent) and Case B (divergent) initialized from $(1.2, 1.2)$ (solid), $(-1.5, -1.5)$ (dash-dot), and $(1.2, 1.5)$ (dashed), respectively.}
	\label{fig:stability}
\end{figure}

We next examine the effectiveness and predictive capability of robust stability criterion \eqref{robust stability criterion} established in Theorem~\ref{Thm: Robust Stability}. \Cref{fig:stability} illustrates the relation between the theoretical stability budget region and the actual closed-loop behavior under two representative cases. As shown in \Cref{fig:stability}(a), the green region characterizes the uncertainty budget for $(c_1,c_2)$ to guarantee the stability preservation based on \Cref{Thm: Robust Stability}. Case A lies strictly inside the theoretically guaranteed stable region, whereas Case B violates the derived robust stability criterion~\eqref{robust stability criterion}, indicating potential instability. The corresponding closed-loop behaviors are reported in \Cref{fig:stability}(b). Consistent with the theoretical prediction, the system under Case A exhibits asymptotic convergence to the origin, while trajectories under Case B become unstable and diverge over time.

\begin{figure}[htbp]
	\centering
	\includegraphics[width=0.75\textwidth]{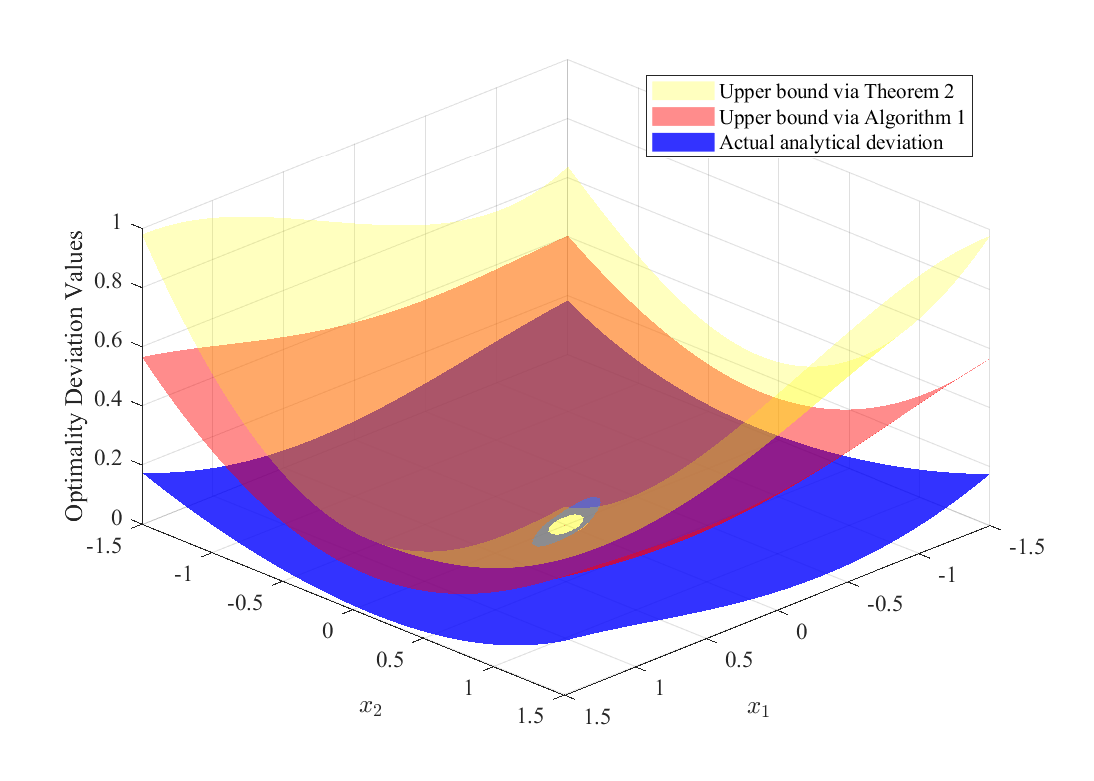} 
	\caption{Validation and hierarchical comparison of the performance deviation bounds over the entire state space. The actual optimality deviation $\|V^*(x) - V_0^*(x)\|$ (blue surface) lies at the bottom of stacked surfaces. The intermediate surface represents the bound $\|V_r^*(x)-V_0^*(x)\|$ constructed via Algorithm~1 (red surface), while the outermost one corresponds to the analytical bound derived in Theorem~\ref{Thm: performance deviation} (yellow surface). The strict ordering structure among the three surfaces confirms the validity of proposed theoretical results.}
	\label{fig:optimality_bound}
\end{figure}

Based on the computed quantities $V_0^*$ and $V_r^*$, the theoretical upper bound for the performance deviation established in \Cref{Thm: performance deviation} is evaluated over the entire state space. Meanwhile, the actual deviation $\|V^*-V_0^*\|$ can be directly computed and compared with the proposed bound, as shown in \Cref{fig:optimality_bound}. The contour plots indicate that the bound successfully captures the spatial distribution of optimality degradation and provides a uniform upper estimate of the actual performance deviation in the considered operating region $\mathbb{X}$. Further, \Cref{fig:controller} visualizes the state-dependent gap between the analytical upper bound and the actual controller deviation, defined as the difference between the right-hand side (RHS) and left-hand side (LHS) of \eqref{controller deviation}. The rigorous non-negativity of the bounding gap over the entire state space $\mathbb{X}$ confirms the validity of Theorem~\ref{Thm: controller deviation}.

\begin{figure}[htbp]
	\centering
	\includegraphics[width=0.7\linewidth]{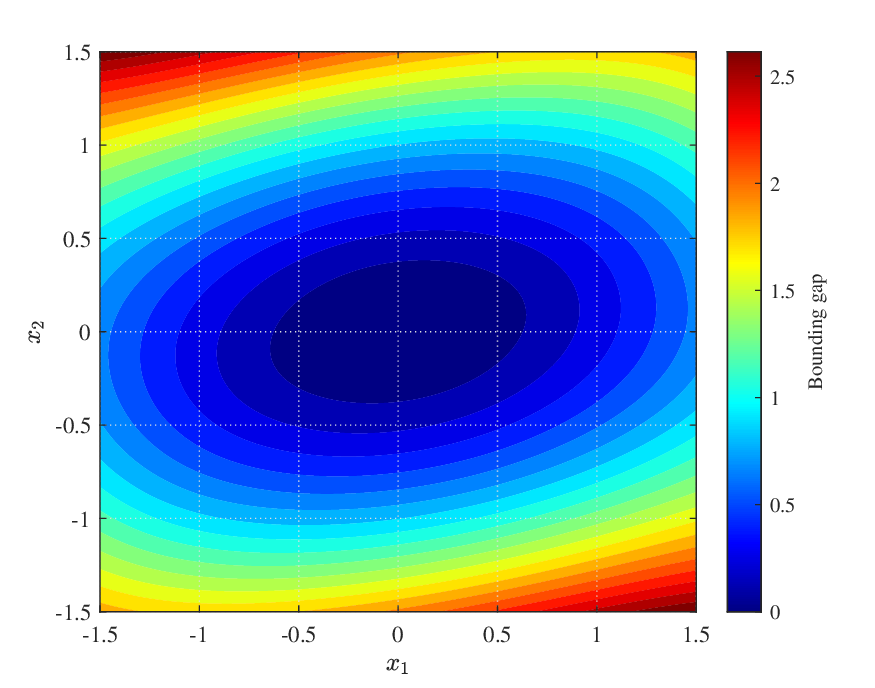}
	\caption{State-dependent bounding gap between the analytical upper bound and the actual controller deviation ($\text{RHS} - \text{LHS}$ of \eqref{controller deviation}), verifying the non-negativity required by Theorem~\ref{Thm: controller deviation}.}
	\label{fig:controller}
\end{figure}

Overall, the simulations consistently support the developed theoretical results. The derived bounds are based on the worst-case uncertainty characterization, which ensures rigorous validity under general nonlinear settings while inherently introduces some conservatism. This reflects the intrinsic gap between worst-case analysis and pointwise system behavior.
\section{Conclusions}\label{7.conclusion}
This paper has established a systematic framework for analyzing the robustness of stability and optimality in nonlinear optimal control. We have demonstrated when closed-loop stability is preserved, how model mismatch affects the achieved performance and the resulting controller, and how the associated optimality deviations can be computationally evaluated through a unified formulation. Furthermore, the LQR specialization reveals a close connection between nonlinear optimality degradation and classical Riccati perturbation analysis, providing additional insights into the underlying robustness mechanism.

Possible directions for future research include extending the present analysis to broader uncertainty structures, incorporating adaptive schemes for controller design, and investigating optimality robustness in more complex and large-scale systems.

\bibliographystyle{siamplain}
\bibliography{references}
\end{document}